\documentclass[11pt]{amsart}

\usepackage{amsthm,amsmath,amsfonts,amssymb}
\usepackage[numbers]{natbib}
\usepackage{tikz}
\usetikzlibrary{trees}
\usepackage{mathtools}
\usepackage{graphicx}
\usepackage{cancel}
\usepackage[margin=1in]{geometry}
\usepackage{hyperref}

\theoremstyle{plain}
\newcounter{thm}
\newtheorem{theorem}[thm]{Theorem}

\newtheorem{remark}[thm]{Remark}
\newtheorem{corollary}[thm]{Corollary}
\newtheorem{proposition}[thm]{Proposition}
\newtheorem{lemma}[thm]{Lemma}
\theoremstyle{definition}

\newtheorem{definition}[thm]{Definition}

\newcommand{\mathbbm}[1]{\boldsymbol{#1}}

\title[VRJP recurrence and fractional-moment decay]{VRJP recurrence and fractional-moment decay for the \(H^{2|2}\) model's effective field on the hierarchical lattice}
\author{Jinglin Wang}
\address{Institut de recherche mathématique avancée, Universit\'e de Strasbourg}
\email{jinglin.wang@unistra.fr}
\author{Xiaolin Zeng}
\address{Institut de recherche mathématique avancée, Universit\'e de Strasbourg}
\email{zeng@math.unistra.fr}
\date{}

\begin{document}

\begin{abstract}
We prove recurrence of the vertex-reinforced jump process on the hierarchical lattice for spectral dimension \(d<2\) for every value of the conductance parameter \(\overline{W}\), and at the critical spectral dimension \(d=2\) for sufficiently strong reinforcement, i.e., sufficiently small \(\overline{W}\). The key estimate is a fractional-moment bound for the Green's function of the associated random Schrödinger operator, expressed as geometric decay across hierarchical scales for the effective \(H^{2|2}\) field. The proof combines the fractional-moment method with an exact hierarchical coarse-graining identity, which turns the path expansion into a recursion over block scales and controls the combinatorial growth created by long-range edges. Together with existing long-range-order results in the regime \(d>2\), these estimates identify the recurrent side of the hierarchical VRJP phase diagram, leaving only the weak-reinforcement critical regime open.
\end{abstract}

\maketitle
\pagestyle{plain}

\section{Introduction}
\label{sec:org669a6eb}
The hierarchical lattice can be realized as a weighted infinite complete graph on the set \(\mathbb{X}=\{1,2,\ldots\}\), on which vertices are grouped into dyadic blocks \(B_{k,m}:=\{2^{k}(m-1)+1 ,\ldots,2^{k}m\}\), called the \(m\)-th dyadic block of order \(k\). Edge weights \(W_{i,j}\) between vertices \(i,j\in \mathbb{X}\) depend only on the hierarchical scale \(d_{\mathbb{X}}(i,j)\), the order of the smallest dyadic block containing both \(i\) and \(j\), and real parameters \(\rho >1\), \(\overline{W}>0\); through the formula
\[W_{i,j}=\overline{W} (2\rho)^{-d_{\mathbb{X}}(i,j)}.\]

In this paper, we study the vertex-reinforced jump process (VRJP) on the hierarchical lattice of spectral dimension defined by \(d=2\frac{\log 2}{\log \rho}\), together with the associated effective \(H^{2|2}\) field, by which we mean the horospherical \(u\)-field arising in the supersymmetric representation of the model (detailed definitions in Section~\ref{h:35597662-da75-4406-b6e2-7bfa14c99d4f}). We prove that the VRJP is recurrent for \(d<2\) and transient for \(d>2\), and that at the critical dimension \(d=2\) it is recurrent in a strong-reinforcement regime, corresponding here to sufficiently small \(\overline{W}\). On the field side, we establish geometric decay across hierarchical scales for fractional moments of the Green's function in the same recurrent regime. Combined with the long-range-order results of \cite{Disertori2022,disertori23:_trans,disertori24:_fluct} for the hierarchical \(H^{2|2}\) model, this supports the interpretation of \(d=2\) as the critical spectral dimension on the hierarchical lattice.

The relevant separation parameter on the hierarchical lattice is not an ordinary graph distance: since the graph is complete, the graph distance between distinct vertices is always one. If one prefers a genuine ultrametric distance, one may use \(r_{\mathbb{X}}(i,j):=2^{d_{\mathbb{X}}(i,j)}\). Our estimates are geometric in the scale variable \(d_{\mathbb{X}}\); equivalently, they can be rewritten as power-law bounds in the ultrametric distance \(r_{\mathbb{X}}\). For this reason, the decay established here should be viewed as a hierarchical-scale statement rather than as Euclidean exponential decay.

It is also useful to keep in mind a binary-tree / \(2\)-adic picture for the geometry. After a suitable relabeling of the vertices by tree coordinates, the hierarchical lattice may be viewed as a dense countable subset of the boundary of the infinite binary tree, equivalently of \(\mathbb Z_{2}\) (2-adics): vertices of \(\mathbb X\) correspond to eventually zero binary rays, whereas general points of \(\mathbb Z_{2}\) correspond to arbitrary infinite rays. In this language, \(d_{\mathbb X}(i,j)\) records the scale of the smallest dyadic block containing \(i\) and \(j\), that is, the first level at which the two rays merge, and \(r_{\mathbb X}(i,j)=2^{d_{\mathbb X}(i,j)}\) is the associated ultrametric distance.

The main results can be summarized informally as follows.

\medskip
\noindent{\bf Theorem (Informal).} For the standard hierarchical weights \(W_{i,j}\), the VRJP is recurrent for \(d<2\) and transient for \(d>2\). At the critical dimension \(d=2\), we prove recurrence in a strong-reinforcement regime. In the same recurrent regimes, the associated effective \(H^{2|2}\) field exhibits geometric decay across hierarchical scales: for every \(0< s< \frac{1}{2}\), the fractional moments of the normalized Green's function satisfy
\[
\mathbb{E}\bigl(e^{s u_i^{[j]}}\bigr)\le C\,\theta^{-s d_{\mathbb{X}}(i,j)}
\]
for suitable constants \(C<\infty\) and \(\theta>1\). In the subcritical case \(d<2\), one may take \(\theta=2\rho\), and this rate is optimal up to the multiplicative constant. More precise statements, including the exact critical bound and the modified critical weights discussed below, are given in Theorems~\ref{thm-vrjp-phase-trans}--\ref{thm-FMM-exp-uij}.
\medskip

Figure \ref{fig-simu} gives a numerical illustration of the change of behavior across spectral dimensions.
\begin{figure}[htbp]
\centering
\includegraphics[scale=0.27]{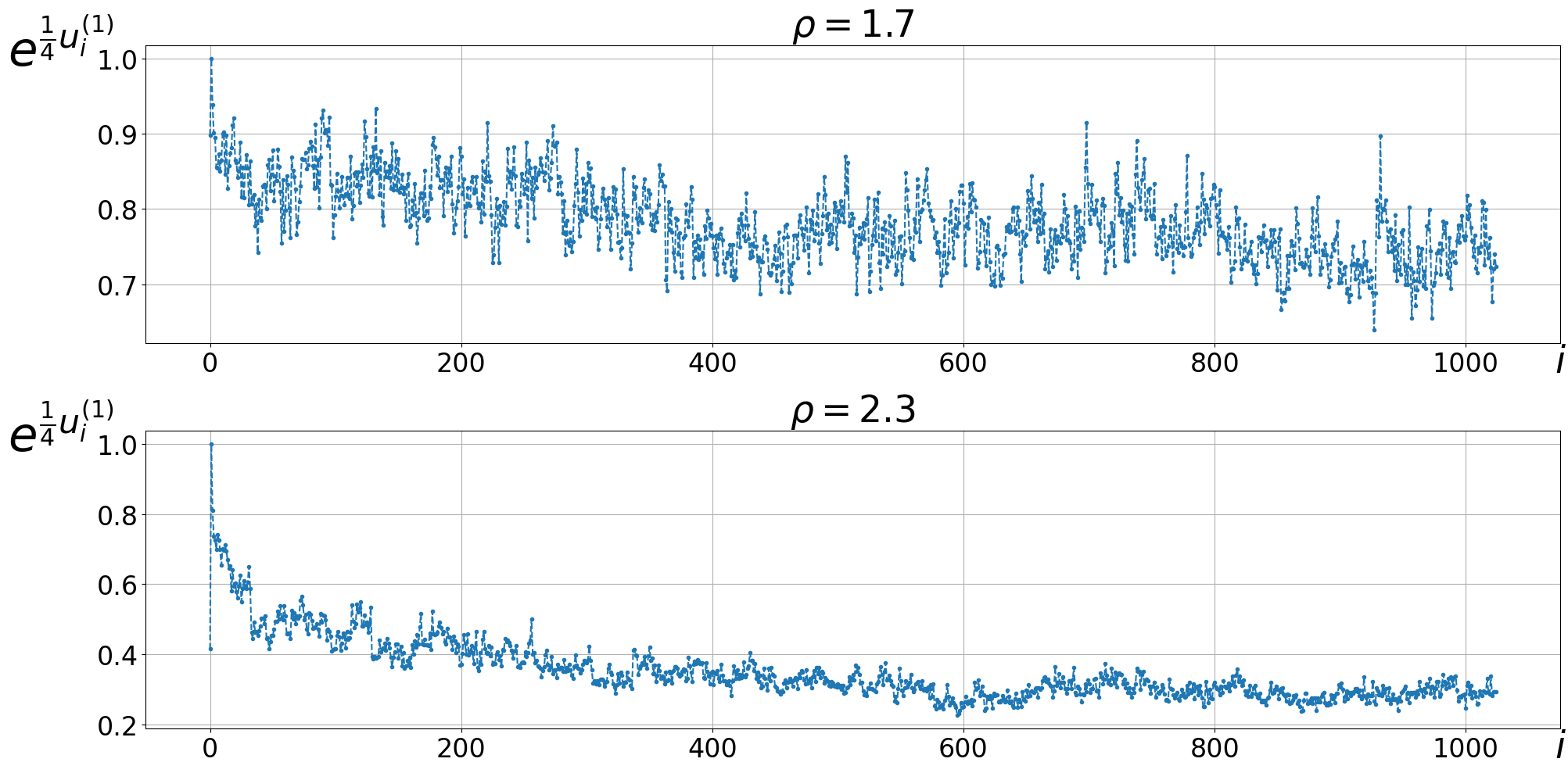}
\caption{\label{fig-simu}Empirical \(1/4\)-moment of \(e^{u^{(1)}_{i}}\) for \(i\in \Lambda_{10}\), based on \(50\) samples on a graph with 1024 vertices, with parameter \(W=1\) and several values of \(\rho\), corresponding to spectral dimensions \(d=2 \frac{\log 2}{\log \rho}\). This plot is illustrative and not intended to be quantitative.}
\end{figure}

The VRJP is a self-interacting jump process whose jump rate from \(i\) to \(j\) is proportional to the conductance \(W_{i,j}\) times the local time already accumulated at \(j\). Strong reinforcement favors repeated visits to previously visited sites, while weak reinforcement makes the process closer to the underlying Markov jump process with jump rate equals \(W_{i,j}\).

On a finite weighted graph, one convenient description of the supersymmetric representation starts from a random Schrödinger operator \(H_{\beta}\) whose off-diagonal entries are \(-W_{i,j}\) and whose diagonal entries are \(2\beta_i\), where the random potential \(\beta\) has an explicit law. Writing \(G=H_{\beta}^{-1}\) and fixing a pinned vertex \(i_{0}\), one then defines the effective field by \(e^{u_i}=G(i_0,i)/G(i_0,i_0)\). This \(u\)-field is the horospherical marginal of the \(H^{2|2}\) model and, in the Sabot--Tarrès representation \cite{Sabot2015,Sabot2019}, it is also the random environment of the standard time-changed VRJP: conditionally on \(u\), the process is a Markov jump process with rates \(\frac12 W_{i,j}e^{u_j-u_i}\); see also \cite[Theorem~3]{Sabot2017}. It is this effective-field/random-environment representation that we use throughout.

On \(\mathbb{Z}^{d}\), this connection has led to a detailed picture of the phase transition for the VRJP. In dimensions \(d\ge 3\), the process has both recurrent and transient regimes depending on the reinforcement strength, while in dimensions \(1\) and \(2\) it is recurrent \cite{Sabot2015,sabot19:_polyn_vertex_reinf_jump_proces,kozma19:_power_vrjp}. In dimension two, finer quantitative decay questions for the associated field are tied to mass-gap-type problems and remain largely open.

On the hierarchical lattice, at the critical spectral dimension \(d=2\), geometric decay in the hierarchical scale \(d_{\mathbb{X}}\), equivalently power-law decay in \(r_{\mathbb{X}}=2^{d_{\mathbb{X}}}\), is the natural hierarchical analogue of a mass-gap-type decay question; compare this with the massive hierarchical Gaussian free field. The hierarchical lattice, introduced by Dyson \cite{Dyson1969,Dyson_1971}, provides a rigid scale structure in which renormalization and coarse graining can be carried out exactly, making it a natural testing ground for such questions.

Recent work \cite{Disertori2022,disertori23:_trans,disertori24:_fluct} established long-range order for the hierarchical \(H^{2|2}\) model when \(d>2\). The present paper complements that picture by treating the recurrent side: for \(d<2\), and for \(d=2\) under strong reinforcement, we prove geometric decay across hierarchical scales for suitable fractional moments and deduce recurrence of the VRJP. Thus, away from the critical point \(d=2\), the phase behavior is governed by the spectral dimension, whereas the critical case is more delicate. We do not address here the weak-reinforcement regime at \(d=2\).

Our proof is based on the fractional-moment method, originally developed in the context of Anderson localization \cite{Aizenman1993}. A direct path expansion on the hierarchical lattice suffers from the large number of long-range edges, whose combinatorial growth weakens standard estimates. The key additional input is an exact coarse-graining identity for the hierarchical \(H^{2|2}\) field, coming from the linear reinforcement structure and developed in \cite{disertori23:_trans}. This identity reduces the problem to a recursion over block scales and controls the path enumeration of the expansion, leading to geometric decay across hierarchical scales for the relevant fractional moments.

The paper is organized as follows. In Section~\ref{h:35597662-da75-4406-b6e2-7bfa14c99d4f}, we define the hierarchical lattice, the VRJP, and the finite-volume random Schrödinger representation, and we state the main results. In Section~\ref{h:9caf97d6-90d8-45b9-a4f1-eca42efe563c}, we prove recurrence and transience of the VRJP using estimates on the associated \(H^{2|2}\) field. Section~\ref{h:6fd27b89-918f-4c37-84a1-e95db8e62cef} contains the proof of the geometric decay estimates via the fractional-moment method and the coarse-graining analysis.
\section{Definition of models and results}
\label{h:35597662-da75-4406-b6e2-7bfa14c99d4f}
In this section, we define the hierarchical lattice, the weighted walk, the finite-volume \(H^{2|2}\) representation, and the precise statements proved later.

The vertex set of the hierarchical lattice is \(\mathbb{X}=\mathbb{N}^{\star}=\{1,2,\dots\}\). For \(k\geq 0\) and \(m\geq 1\), let
\[
B_{k,m}:=\{2^k(m-1)+1,\dots,m2^k\}.
\]
The sets \(B_{k,m}\) are the dyadic blocks of order \(k\). For \(i,j\in \mathbb{X}\), the hierarchical distance is defined by
\[
d_{\mathbb{X}}(i,j):=\min\{k\geq 0:\exists m\geq 1 \text{ such that } i,j\in B_{k,m}\}.
\]
Thus \(d_{\mathbb{X}}(i,j)\) is the first scale at which \(i\) and \(j\) belong to a common block. For example, \(d_{\mathbb{X}}(1,1)=0\), \(d_{\mathbb{X}}(1,7)=3\), and \(d_{\mathbb{X}}(i,2^n+1)=n+1\) for every \(1\le i\le 2^n\). The associated ultrametric distance is \(2^{d_{\mathbb{X}}(i,j)}\), but throughout the paper we work with the scale index \(d_{\mathbb{X}}\). We view \(\mathbb{X}\) as a complete graph: every pair of distinct vertices \(i,j\in \mathbb{X}\) is connected by an edge \(\{i,j\}\).

Fix real parameters \(\rho>1\) and \(\overline{W}>0\). For distinct vertices \(i,j\in \mathbb{X}\), we define the edge weights by
\begin{equation}\begin{aligned}
\label{eq-def-of-Wij}
W_{i,j}=\overline{W}(2\rho)^{-d_{\mathbb{X}}(i,j)}, \qquad W_{i,i}=0.
\end{aligned}\end{equation}
For every vertex \(i\), the total conductance is finite:
\[
\sum_{j\in \mathbb{X}}W_{i,j}=\overline{W}\sum_{k=1}^{\infty}2^{k-1}(2\rho)^{-k}=\frac{\overline{W}}{2(\rho-1)}<\infty.
\]
We also use the bounded self-adjoint operator \(P_W\) on \(\ell^2(\mathbb{X})\) defined by
\[
P_Wf(i)=\sum_{j\in \mathbb{X}}W_{i,j}f(j), \qquad f\in \ell^2(\mathbb{X}).
\]
The conductance-weighted random walk associated with \((W_{i,j})\) is the continuous-time Markov chain that jumps from \(i\) to \(j\) at rate \(W_{i,j}\); equivalently, its generator is
\[
L_Wf(i)=\sum_{j\ne i}W_{i,j}\bigl(f(j)-f(i)\bigr).
\]
The spectral dimension of \((\mathbb{X},(W_{i,j}))\) is
\[
d=2\frac{\log 2}{\log \rho},
\]
and the underlying conductance-weighted walk has a phase transition at \(d=2\); see \cite[Eq.~(3.1) and Theorem~3.1]{Kritchevski_2008}.

We now define the vertex-reinforced jump process (VRJP) \((Y_t)_{t\ge 0}\) on the infinite weighted graph \(\mathbb{X}\). Starting from \(Y_0=i_0\), when the process is at vertex \(i\) at time \(t\), it jumps to \(j\ne i\) at rate
\begin{equation}\begin{aligned}
\label{eq-jumprate-vrjp}
W_{i,j}\left(1+\int_0^t \mathbbm{1}_{Y_s=j}\,ds\right).
\end{aligned}\end{equation}
Although the graph is not locally finite, the model is well defined because \(\sum_{j\in \mathbb{X}}W_{i,j} < \infty\) for all \(i \in \mathbb{X}\), and the infinite-volume mixture construction applies in this setting; see \cite[Proposition~1 and Theorem~1(iii)]{Sabot2019} and the discussion on p.~2 of \cite{disertori23:_trans}. Up to a random time change, the VRJP is a mixture of Markov jump processes.

The associated random Schrödinger representation of the VRJP \((Y_{t})_{t\ge 0}\) is encoded by a positive real random field \((\beta_i)_{i\in \mathbb{X}}\). For every finite set \(U\subset \mathbb{X}\), the marginal law of \((\beta_i)_{i\in U}\) is characterized by the Laplace transform
\begin{equation}\begin{aligned}
\label{eq-laplace-transform-beta}
\mathbb{E}\left(e^{-\sum_{i\in U}\lambda_i\beta_i}\right)=e^{-\sum_{\{i,j\}\subset U}W_{i,j}(\sqrt{(1+\lambda_i)(1+\lambda_j)}-1)-\sum_{i\in U,\,j\notin U}W_{i,j}(\sqrt{1+\lambda_i}-1)}\prod_{i\in U}\frac{1}{\sqrt{1+\lambda_i}}.
\end{aligned}\end{equation}
The second exponential term is finite because \(\sum_{j\notin U}W_{i,j}<\infty\) for every \(i\in U\). Existence of the field follows from \cite[Proposition~1]{Sabot2019}. This field gives rise to the random Schrödinger operator
\[
H_\beta=2\beta-P_W,
\]
where \(2\beta\) denotes the multiplication operator by \((2\beta_i)_{i\in \mathbb{X}}\), so that
\[
H_\beta f(i)=2\beta_i f(i)-\sum_{j:j\ne i}W_{i,j}f(j), \qquad f\in \ell^2(\mathbb{X}).
\]
Almost surely, \(H_\beta\) defines a positive self-adjoint operator on its natural \(\ell^2(\mathbb X)\) domain. For the purposes of this paper, however, we only use its finite-volume restrictions.

\subsection{Finite box approximations}
\label{sec:org74480ba}
We now pass to finite-volume approximations with a wired boundary vertex.

For \(n\geq 1\), let \(\Lambda_n=\{1,2,\dots,2^n\}\) and add one extra vertex \(\delta_n\). We write \(\widetilde{\Lambda}_n:=\Lambda_n\cup\{\delta_n\}\), where \(\delta_n\) represents the exterior \(\mathbb{X}\setminus \Lambda_n\). For \(i\in \Lambda_n\), using \eqref{eq-def-of-Wij}, define the boundary weight by
\begin{equation}\begin{aligned}
\label{eq-edgeweight}
W_{\delta_n,i}:=\sum_{j\notin \Lambda_n}W_{i,j}
=\sum_{m\ge 1}2^{n+m-1}\,\overline{W}(2\rho)^{-(n+m)}
=\frac{\overline{W}\rho^{-n}}{2(\rho-1)}.
\end{aligned}\end{equation}
Here the second equality uses that, for each \(m\ge 1\), there are \(2^{n+m-1}\) vertices outside \(\Lambda_n\) at hierarchical distance \(n+m\) from \(i\).

By \cite[Lemma~2]{Sabot2019}, one can define a random variable \(\beta_{\delta_n}\) so that the law of \((\beta_i)_{i\in \widetilde{\Lambda}_n}\) has density
\begin{equation}\begin{aligned}
\label{eq-beta-density-Lambdatilda-n}
\left(\sqrt{\frac{2}{\pi}}\right)^{|\widetilde{\Lambda}_n|}\mathbbm{1}_{H_{\beta,\widetilde{\Lambda}_n}>0}e^{-\frac12\langle 1,H_{\beta,\widetilde{\Lambda}_n}1\rangle}\frac{1}{\sqrt{\det H_{\beta,\widetilde{\Lambda}_n}}}\prod_{i\in \widetilde{\Lambda}_n}d\beta_i,
\end{aligned}\end{equation}
where \(H_{\beta,\widetilde{\Lambda}_n}\) acts on \(\ell^2(\widetilde{\Lambda}_n)\) by
\[
H_{\beta,\widetilde{\Lambda}_n}f(i)=2\beta_i f(i)-\sum_{j\in \widetilde{\Lambda}\setminus \{ i\} }W_{i,j}f(j), \qquad i\in \widetilde{\Lambda}_n.
\]
Thus \(H_{\beta,\widetilde{\Lambda}_n}\) is the random Schrödinger matrix associated with the \(H^{2|2}\) measure on \(\widetilde{\Lambda}_{n}\).

The finite-volume model approximates the infinite-volume VRJP in the following sense. Let \(\mathbb{P}^{\mathbb{X}}_1\) denote the law of the discrete-time chain associated with the VRJP on \(\mathbb{X}\) started from \(1\in \mathbb{X}\), and let \(\mathbb{P}^{\widetilde{\Lambda}_n}_1\) be the corresponding law on \(\widetilde{\Lambda}_n\). If \(\tau_1^+\) is the first return time to \(1\) and \(\tau_{\delta_n}\) the first hitting time of \(\delta_n\), then Section~6.2 of \cite{Sabot2019} yields
\[
\lim_{n\to\infty}\mathbb{P}^{\widetilde{\Lambda}_n}_1(\tau_1^+<\tau_{\delta_n})=\mathbb{P}^{\mathbb{X}}_1(\tau_1^+<\infty).
\]
Hence recurrence and transience of the infinite-volume VRJP can be studied through these finite-volume approximations.

It is often more convenient to replace the variables \((\beta_i)_{i\in \widetilde{\Lambda}_n}\) by the effective \(u\)-field and a boundary variable. Let
\[
G_{\beta,\widetilde{\Lambda}_n}:=H_{\beta,\widetilde{\Lambda}_n}^{-1}, \qquad \gamma_n:=\frac{1}{2G_{\beta,\widetilde{\Lambda}_n}(\delta_n,\delta_n)},
\]
and define
\begin{equation}\begin{aligned}
\label{eq-def-u-i-n}
e^{u_i^{(n)}}=\frac{G_{\beta,\widetilde{\Lambda}_n}(\delta_n,i)}{G_{\beta,\widetilde{\Lambda}_n}(\delta_n,\delta_n)}, \qquad i\in \Lambda_n.
\end{aligned}\end{equation}
By convention, \(u_{\delta_n}^{(n)}=0\). Here the superscript \((n)\) records pinning at \(\delta_n\); later, for pinning at a general vertex \(j\), we will write \(u_i^{[j]}\).
The family \((u_i^{(n)})_{i\in \Lambda_n}\) is called the effective \(u\)-field of the pinned \(H^{2|2}\) model in e.g. \cite{Sabot2019}.

On finite graphs, we now explain that this same field is also the random environment of the standard time-changed VRJP \cite[Theorem~3]{Sabot2017}. Consider the VRJP on \(\widetilde{\Lambda}_{n}\) with \(Y_{0}=\delta_{n}\). Let \(L_{j}(t)=1+\int_{0}^{t} \mathbbm{1}_{Y_{s}=j}ds\) and \(S_{j}(t)=L_{j}(t)^{2}-1\). The time change \(D(t)=\sum_{i\in \widetilde{\Lambda}_{n}}(L_{i}(t)^{2}-1)\) is strictly increasing, and we set \(Z_{t}=Y_{D^{-1}(t)}\).
\[U_{i}^{(n)} := \frac{1}{2} \lim_{t\to \infty}\left(\log (S_{i}(t)+1)-\log(S_{\delta_{n}}(t)+1) \right)\]
exists almost surely for all \(i\in \widetilde{\Lambda}_{n}\) and the law of \((U_{i}^{(n)})_{i\in \widetilde{\Lambda}_{n}}\) is explicit (see, e.g., Eq. (1.4) of \cite{disertori23:_trans}). \((U_{i}^{(n)})_{i\in \Lambda_{n}}\) is equal in law to the effective field \((u_{i}^{(n)})_{i\in \Lambda_{n}}\) of the \(H^{2|2}\) model on \(\widetilde{\Lambda}_{n}\). In the sequel we write \((u_{i}^{(n)})\) instead of \((U_{i}^{(n)})\) when there is no ambiguity.

Conditionally on \((u_i^{(n)})\), the process \((Z_{t})_{t\ge 0}\) is a Markov jump process with jump rates \(\frac12 W_{i,j}e^{u_j^{(n)}-u_i^{(n)}}\) from \(i\) to \(j\); see \cite[Theorem~2]{Sabot2015} and \cite[Theorem~3]{Sabot2017}. Moreover, \cite[Eq.~(4.2)]{Sabot2019} shows that \((\beta_i)_{i\in \widetilde{\Lambda}_n}\) can be recovered from \((u_i^{(n)})_{i\in \widetilde{\Lambda}_n}\) and \(G_{\beta,\widetilde{\Lambda}_n}(\delta_n,\delta_n)\) via
\[
2\beta_i=\sum_{j\in \widetilde{\Lambda}_n}W_{i,j}e^{u_j^{(n)}-u_i^{(n)}}+\mathbbm{1}_{i=\delta_n}\frac{1}{G_{\beta,\widetilde{\Lambda}_n}(\delta_n,\delta_n)}, \qquad i\in \widetilde{\Lambda}_n.
\]
An important consequence of supersymmetry is the Ward identity
\begin{equation}\begin{aligned}
\label{eq-ward-identity-eu=1}
\mathbb{E}(e^{u_i^{(n)}})=1, \qquad i\in \widetilde{\Lambda}_n,
\end{aligned}\end{equation}
see for instance \cite[Eq.~(5.26)]{Disertori2017a}.

We will repeatedly use the following standard lemma; see \cite[Lemma~5.2]{Collevecchio2018} and \cite[Theorem~3]{Sabot2017}.
\begin{lemma}
  \label{lem1}
Let \(\mathcal{G}=(V,E,W)\) be a finite weighted undirected graph with edge weights \(W_{i,j}\). Let \(H_\beta\) be the random Schrödinger matrix defined by \(H_\beta(i,j)=-W_{i,j}\) for \(i\ne j\) and \(H_\beta(i,i)=2\beta_i\), where the law of \(\beta\) has density
\[
\left(\sqrt{\frac{2}{\pi}}\right)^{|V|}\mathbbm{1}_{H_\beta>0}e^{-\frac12\langle 1,H_\beta 1\rangle}\frac{1}{\sqrt{\det H_\beta}}\prod_{i\in V}d\beta_i.
\]
We call \(H_\beta\) the random Schrödinger matrix associated with the finite-volume \(H^{2|2}\) measure on \(\mathcal{G}\).

For any \(i\in V\), the law of \(\frac{1}{2 G_{\beta}(i,i)}\) is as follows: for all \(a>0\),
\begin{equation}\begin{aligned}
\label{eq-density-G-i-i}
\mathbb{P}\left( (2G_{\beta}(i,i))^{-1} <a\right) = \int_{0}^{a} \frac{1}{\sqrt{\pi t}} e^{-t}dt
\end{aligned}\end{equation}
That is, \((2G_{\beta}(i,i))^{-1}\) has Gamma distribution with shape \(1/2\) and rate \(1\). For any \(0 < s < \frac{1}{2}\), \(\mathbb{E}( G_{\beta}(i,i)^{s}) < \infty\). Moreover, \(G_{\beta}(i,i)\) is independent of the pinned field \(\{u_{j}^{(i)}: j\ne i\}\).

For \(i \in V\), define the effective \(u\)-field pinned at \(i\) by
\[e^{u^{(i)}_{j}}= \frac{G_{\beta}(i,j)}{G_{\beta}(i,i)},\ G_{\beta}=H_{\beta}^{-1},\]

\label{lem-gamma-dist-Gii}
\end{lemma}

In particular, for every \(0 < s < \frac12\),
\begin{equation}\begin{aligned}
\label{eq-euij-euji-are-equal}
\mathbb{E}(e^{s u_j^{(i)}})=\mathbb{E}(e^{s u_i^{(j)}}).
\end{aligned}\end{equation}
Indeed,
\[
\mathbb{E}(G_\beta(i,i)^s)\mathbb{E}(e^{s u_j^{(i)}})=\mathbb{E}(G_\beta(i,j)^s)=\mathbb{E}(G_\beta(j,i)^s)=\mathbb{E}(G_\beta(j,j)^s)\mathbb{E}(e^{s u_i^{(j)}}).
\]

Almost surely, \(H_{\beta,\widetilde{\Lambda}_n}\) is an M-matrix (see \cite[Section~6]{Sabot2017}), hence \(G_{\beta,\widetilde{\Lambda}_n}(i,j)>0\) for all \(i,j\). When there is no ambiguity, we simply write \(u_i\) instead of \(u_i^{(n)}\).

By Lemma \ref{lem1}, one can pin the field at any vertex \(j\), not only at \(\delta_n\). We denote the field pinned at \(j\) by \((u_i^{[j]}){i\in\widetilde\Lambda_n}\), where
\begin{equation}
  \label{eq:uijpinatj}
  e^{u_{i}^{[j]}}=\frac{G_{\beta,\widetilde{\Lambda}_n}(j,i)}{G_{\beta,\widetilde{\Lambda}_n}(j,j)}, \qquad i\in \widetilde\Lambda_n. 
\end{equation}
Throughout the sequel, the ambient graph \(\widetilde\Lambda_n\) is fixed by context, so we suppress the superscript \((n)\) unless we need to emphasize the boundary pinning at \(\delta_n\), which we denote by \(u_i^{(n)}\).

Since \(H_{\beta,\widetilde{\Lambda}_n}\) is almost surely positive definite, the map
\[
(\beta_i)_{i\in \widetilde{\Lambda}_n}\longmapsto \bigl((u_i)_{i\in \Lambda_n},\gamma_n\bigr)
\]
is a diffeomorphism; see \cite[Section~6.3]{Sabot2017}. In the random-environment representation of the time-changed VRJP \((Z_{t})_{t\ge 0}\), the associated discrete-time chain of \((Z_{t})_{t\ge 0}\) is a random conductance model with conductances
\[
W_{i,j}e^{u_i+u_j}.
\]
Thus the fluctuations of the \(u\)-field determine the long term behaviors of the reinforced walk.

\subsection{Results and strategies of proofs}
\label{sec:org2dbd951}
We now state the precise results proved in this paper.

\begin{theorem}
Consider the VRJP \((Y_t)_{t\ge 0}\) on the hierarchical lattice \((\mathbb{X},(W_{i,j}))\), where the weights are given by \eqref{eq-def-of-Wij}, started from an arbitrary vertex.
\begin{enumerate}
\item If \(d>2\), then almost surely every vertex is visited only finitely many times; in particular, the VRJP is transient.
\item If \(d<2\), then almost surely every vertex is visited infinitely many times; in particular, the VRJP is recurrent.
\item If \(d=2\) and \(\overline{W}\) is sufficiently small, then the VRJP is recurrent.
\end{enumerate}
\label{thm-vrjp-phase-trans}
\end{theorem}

\begin{remark}
At the critical value \(\rho=2\), following \cite{disertori23:_trans}, one may also consider the modified weights
\begin{equation}\begin{aligned}
\label{eq-def-Wij-2d}
W_{i,j}=\overline{W}4^{-d_{\mathbb{X}}(i,j)}Q\bigl(d_{\mathbb{X}}(i,j)\bigr),
\end{aligned}\end{equation}
where \(Q:\mathbb{N}_{\ge 1}\to (0,\infty)\) is assumed to satisfy: for some \(\alpha > 1\), there are \(c_{1},c_{2} >0\) and \(n_{0} \in \mathbb{N}\) such that for all \(n \ge n_{0}\),
\begin{equation}
  \label{eq:Qbound}
 c_{1} n^{\alpha} \le Q(n) \le c_{2} n^{\alpha}.
\end{equation}
Such a factor changes the critical decay only by a polynomial correction; heuristically, polynomial growth of \(Q\) pushes the model toward the transient side, whereas polynomial decay points in the opposite direction. When \(\alpha > 1\), an effective-resistance argument similar to Section~2 of \cite{disertori23:_trans} shows that the wired finite-volume resistances remain bounded, so the underlying conductance network is transient, which suggests transience of the VRJP with the edge weights in \eqref{eq-def-Wij-2d}. We record this only as background for a related modified critical regime; the present paper does not treat this variant further.
\end{remark}
Our next results are fractional-moment estimates for the normalized Green's function of the finite-volume random Schrödinger matrix. We first state the estimate for the field pinned at \(\delta_n\).
\begin{theorem}
Assume the weights \(W_{i,j}\) are given by \eqref{eq-def-of-Wij}. Let \(u^{(n)}\) be defined by \eqref{eq-def-u-i-n}, and fix \(0 < s < \frac12\). Then the following fractional-moment estimates hold uniformly in \(n\ge 1\). 
\begin{enumerate}
\item Suppose \(d<2\). For all \(\overline{W} > 0\), there exists \(C(\overline{W},d,s)>0\) such that, for every \(i\in \Lambda_n\),
\begin{equation}\begin{aligned}
\label{eq-FMM-eu-subcritical}
\mathbb{E}\bigl(e^{s u_i^{(n)}}\bigr)\le C(\overline{W},d,s)\rho^{-sn}.
\end{aligned}\end{equation}
\item Suppose \(d=2\), and \(\overline{W}\) is sufficiently small. Then there exists \(C(\overline{W},s)>0\) such that, for every \(i\in \Lambda_n\),
\begin{equation}\begin{aligned}
\label{eq-FMM-eu-critical}
\mathbb{E}\bigl(e^{s u_i^{(n)}}\bigr)\le C(\overline{W},s)c(\overline{W},s)^{-sn},
\end{aligned}\end{equation}
where
\[
c(\overline{W},s)=\frac{2}{\left(1+\left(1+\left(\frac{c_s\overline{W}}{4}\right)^s\right)\left(\frac{c_s\overline{W}}{4}\right)^s\right)^{1/s}},
\]
and
\[c_{s}^{s} = \mathbb{E}(G_{\beta}(i,i)^{s}) =  \frac{1}{2^{s}\sqrt{\pi}} \Gamma\left(\frac{1}{2}-s\right). \]
In particular, \(c(\overline{W},s)\to 2\) as \(\overline{W}\to 0\), so the bound decays geometrically in the scale \(n\) when \(\overline{W}\) is small enough such that \(c(\overline{W},s) > 1\).
\end{enumerate}
\label{thm-FMM-exp-u}
\end{theorem}

We also obtain the corresponding estimate for the field pinned at an arbitrary vertex.
\begin{theorem}
Fix \(0 < s < \frac12\) and \(n\ge 1\). Assume the weights \(W_{i,j}\) are given by \eqref{eq-def-of-Wij}. Recall notation \( u_i^{[j]}\) defined in \eqref{eq:uijpinatj}.
\begin{enumerate}
\item Suppose \(d<2\). For all \(\overline{W} >0\), there exists \(C(\overline{W},d,s)>0\) such that, for every \(i,j\in \Lambda_n\),
\begin{equation}\begin{aligned}
\label{eq-FMM-eu-subcritical-ij}
\mathbb{E}\bigl(e^{s u_i^{[j]}}\bigr)\le C(\overline{W},d,s)(2\rho)^{-s d_{\mathbb{X}}(i,j)}.
\end{aligned}\end{equation}
\item Suppose \(d=2\), and \(\overline{W}\) is sufficiently small. Then there exists \(C(\overline{W},s)>0\) such that, for every \(i,j\in \Lambda_n\),
\begin{equation}\begin{aligned}
\label{eq-FMM-eu-critical-ij}
\mathbb{E}\bigl(e^{s u_i^{[j]}}\bigr)\le C(\overline{W},s)\bigl(2c(\overline{W},s)\bigr)^{-s d_{\mathbb{X}}(i,j)},
\end{aligned}\end{equation}
where \(c(\overline{W},s)\) is the same constant as in Theorem~\ref{thm-FMM-exp-u}.
\end{enumerate}
\label{thm-FMM-exp-uij}
\end{theorem}

The meaning of these estimates is different from the usual Euclidean one. On \(\mathbb{Z}^d\), the fractional-moment method produces exponential decay of \(\mathbb{E}(G(i,j)^{s})\) in the Euclidean distance between \(i,j\). On the hierarchical lattice, by contrast, there is always a direct edge between \(i\) and \(j\), with weight of order \((2\rho)^{-d_{\mathbb{X}}(i,j)}\). Thus the natural scale variable is \(d_{\mathbb{X}}(i,j)\), and the direct one-edge contribution already gives the lower bound
\[
\mathbb{E}\bigl(G(i,j)^s\bigr)\gtrsim (2\rho)^{-s d_{\mathbb{X}}(i,j)}.
\]
In the subcritical regime \(d<2\), the upper bounds of Theorems~\ref{thm-FMM-exp-u} and \ref{thm-FMM-exp-uij} therefore match the natural lower bound up to a multiplicative constant, so the hierarchical-scale decay rate is optimal.

A naive path expansion is nevertheless insufficient, because the number of self-avoiding walks grows very rapidly on the complete hierarchical graph. The key additional input here is an exact coarse-graining property of the \(H^{2|2}\) field: if a set of vertices looks identical from the outside, then it can be merged into a single effective vertex without changing the form of the field distribution. Combined with uniform fractional moments of diagonal Green's functions, this yields a recursion over block scales and controls the number of paths in the expansion.
\section{Recurrence and transience of the VRJP, proof of Theorem \ref{thm-vrjp-phase-trans}}
\label{h:9caf97d6-90d8-45b9-a4f1-eca42efe563c}
\subsection{Recurrence of the VRJP in the recurrent regimes}
\label{sec:orgc39258d}
Consider the VRJP \(Y_t\) on \(\mathbb X\) started from \(1\), and let \(T_n\) be the first exit time from \(\Lambda_n\). Up to time \(T_n^-\), the process has the same law as the finite-volume VRJP \(\widetilde Y_t^{(n)}\) on \(\widetilde\Lambda_n\), started from \(1\), up to its first visit to \(\delta_n\). This provides a coupling of VRJPs starting from 1 on graphs \(\widetilde{\Lambda}_{n},\ n\ge 1\) with the VRJP on \(\mathbb X\), up to time \(T_{n}^{-}\).

By \cite[Theorem~1(iii)]{Sabot2019}, after a random time change, \(\widetilde Y_t^{(n)}\) is a mixture of Markov jump processes, and its associated discrete-time chain is a random conductance model with conductances
\[
C^{(n)}_{i,j}=W_{i,j}e^{u_i^{(1)}+u_j^{(1)}},
\]
where \((u_i^{(1)})_{i\in \widetilde\Lambda_n}\) is the effective field of the pinned \(H^{2|2}\) model on \(\widetilde\Lambda_n\), defined by
\begin{equation}\begin{aligned}
\label{eq-e-u-1i-Lambdatildan}
e^{u^{(1)}_{i}} = \frac{G_{\beta,\widetilde{\Lambda}_{n}}(1,i)}{G_{\beta,\widetilde{\Lambda}_{n}}(1,1)},\ \ i\in \widetilde{\Lambda}_{n}.
\end{aligned}\end{equation}
Here \(G_{\beta,\widetilde\Lambda_n}\) is the same Green's function as in \eqref{eq-def-u-i-n}.

Let \(\tau_1^+\) be the first return time to \(1\) of \(\widetilde Y_t^{(n)}\). Let \({P}_{1}^{\widetilde{\Lambda}_{n}}\) be the quenched probability, that is, the law of \((\widetilde{Y}_{t}^{(n)})\) given \((u_{i}^{(1)})_{i\in \widetilde{\Lambda}_{n}}\). By \cite[Eq.~(2.4)]{Lyons2015},
\begin{equation}
\label{eq-hit-prob-ceff}
{P}_{1}^{\widetilde{\Lambda}_{n}}(\tau_{\delta_n}<\tau_1^+) = \frac{1}{\pi_1^{(n)}} C_{\mathrm{eff}}^{(n)}(1,\delta_n),
\end{equation}
where \(\pi_1^{(n)}=\sum_{j\in \widetilde\Lambda_n} W_{1,j}e^{u_j^{(1)}}\), and \(C_{\mathrm{eff}}^{(n)}(1,\delta_n)\) is the effective conductance of the electrical network with conductances \(C_{i,j}^{(n)}\) on the finite graph \(\widetilde\Lambda_n\).

By Rayleigh's monotonicity principle (\cite[Section~2.4]{Lyons2015}), for \(i,j\in \Lambda_n\), set \(W_{i,j}^+=\infty\), and denote by a superscript \(+\) the quantities computed from these enlarged conductances. Then
\[
C_{\mathrm{eff}}^{(n)}(\delta_n,1)\le C_{\mathrm{eff}}^{(n),+}(\delta_n,1).
\]
Since \(e^{u_i^{(1)}+u_j^{(1)}}>0\) almost surely for every \(i,j\in \Lambda_n\), we have \(C_{i,j}^+=\infty\) on \(\Lambda_n\). Thus all vertices in \(\Lambda_n\) can be identified, and the resulting network between \(1\) and \(\delta_n\) is a parallel one. Therefore
\begin{equation}
\label{eq-ceff-collapsed}
C_{\mathrm{eff}}^{(n),+}(\delta_n,1)=\sum_{i\in \Lambda_n}C_{\delta_n,i}=\sum_{i\in \Lambda_n}W_{\delta_n,i}e^{u_i^{(1)}+u_{\delta_n}^{(1)}}.
\end{equation}
There are two cases: \(d<2\), and \(d=2\) with \(Q\equiv 1\) and \(\overline W\) sufficiently small. We first treat \(d<2\). Fix \(0 < s < \frac12\) so close to \(\frac12\) that \(2(2\rho)^{-s}<1\). By \eqref{eq-edgeweight},
\[
W_{\delta_n,i}=\frac{\overline W\rho^{-n}}{2(\rho-1)},
\]
so
\[
C_{\mathrm{eff}}^{(n),+}(\delta_n,1)=W_{\delta_n,i}e^{u_{\delta_n}^{(1)}}\sum_{i\in \Lambda_n}e^{u_i^{(1)}}.
\]
Using \eqref{eq-FMM-power-bound} and Theorem~\ref{thm-FMM-exp-uij}, we obtain
\[
\begin{aligned}
\mathbb P\left(\sum_{i\in \Lambda_n}e^{u_i^{(1)}}>\chi_1\right)
&\le \chi_1^{-s}\,\mathbb E\left(\left(\sum_{i\in \Lambda_n}e^{u_i^{(1)}}\right)^s\right) \\
&\le \chi_1^{-s}\sum_{i\in \Lambda_n}\mathbb E\bigl(e^{s u_i^{(1)}}\bigr) \\
&\le C(\overline W,d,s)\,\chi_1^{-s}\sum_{k=0}^n 2^k(2\rho)^{-sk}
\le C_1\chi_1^{-s},
\end{aligned}
\]
where \(C_1<\infty\) because \(2(2\rho)^{-s}<1\). Also, by the Ward identity \eqref{eq-ward-identity-eu=1},
\[
\mathbb P\bigl(e^{u_{\delta_n}^{(1)}}>\chi_2\bigr)\le \chi_2^{-1}.
\]

Next we bound \(\pi_1^{(n)}\) from below. By \cite[Proposition~6]{Sabot2019},
\[
e^{u_2^{(1)}}>\frac{W_{1,2}}{2\beta_2},
\]
therefore
\[
\pi_1^{(n)}>W_{1,2}e^{u_2^{(1)}}>\frac{W_{1,2}^2}{2\beta_2}.
\]
By \cite[Proposition~1]{Sabot2019}, the variable \(2\beta_2\) has an inverse-Gaussian law with parameters depending only on \(W_2=\sum_{i\in\mathbb X}W_{2,i}\), hence its upper tail is exponentially small. Consequently, there exist constants \(c_0,C_0>0\), independent of \(n\), such that for all \(\chi_3>0\),
\[
\mathbb P\bigl(\pi_1^{(n)}<\chi_3\bigr)
\le \mathbb P\left(\frac{W_{1,2}^2}{2\beta_2}<\chi_3\right)
= \mathbb P\left(2\beta_2>\frac{W_{1,2}^2}{\chi_3}\right)
\le C_0 e^{-c_0/\chi_3}.
\]

On the event
\[
\left\{\sum_{i\in \Lambda_n}e^{u_i^{(1)}}\le \chi_1\right\}\cap
\left\{e^{u_{\delta_n}^{(1)}}\le \chi_2\right\}\cap
\left\{\pi_1^{(n)}\ge \chi_3\right\},
\]
we have
\[
 P_1^{\widetilde\Lambda_n}(\tau_{\delta_n}<\tau_1^+)
\le \frac{C_{\mathrm{eff}}^{(n),+}(\delta_n,1)}{\pi_1^{(n)}}
\le \frac{\overline W}{2(\rho-1)}\rho^{-n}\frac{\chi_1\chi_2}{\chi_3}.
\]
Choose
\[
\chi_1=\rho^{n/6},\qquad \chi_2=\rho^{n/6},\qquad \chi_3=\rho^{-n/6}.
\]
Then there exists \(C_2>0\) such that
\[
\mathbb P\left( P_1^{\widetilde\Lambda_n}(\tau_{\delta_n}<\tau_1^+)\le C_2\rho^{-n/2}\right)
\ge 1-C_0e^{-c_0\rho^{n/6}}-C_1\rho^{-ns/6}-\rho^{-n/6}.
\]
The probabilities of complementary events are summable in \(n\), and the quenched VRJPs on \(\widetilde{\Lambda}_{n}\) up to time \(\tau_{\delta_{n}}\) are coupled, so by Borel--Cantelli, almost surely,
\[
\exists n_{0}, \forall n\ge n_{0}, P_{1}^{\widetilde{\Lambda}_{n}}(\tau_{\delta_{n}} < \tau_{1}^{+}) < C_{2} \rho^{-n/2}.
\]
Therefore, using the coupling, almost surely,
\[
\exists n_{0}, \forall n\ge n_{0}, P_{1}^{\mathbb{X}}(\tau_{\Lambda_{n}^{c}} > \tau_{1}^{+}) \ge  1- C_{2} \rho^{-n/2}.
\]

Passing to the limit \(n\to\infty\) gives
\[
 P_1^{\mathbb X}(\tau_1^+<\infty)=1.
\]
By Theorem 1.(iii) of \cite{Sabot2019}, the VRJP (up to time change) on \(\mathbb{X}\) is a random conductance model, thus recurrence of one state implies infinitely many returns to that state, which proves recurrence for \(d<2\).

When \(d=2\), \(Q\equiv 1\), and \(\overline W\) is sufficiently small, the same estimate gives
\[
\mathbb P\left(\sum_{i\in \Lambda_n}e^{u_i^{(1)}}>\chi_1\right)
\le C(\overline W,s)\chi_1^{-s}\sum_{k=0}^n \lambda^k,
\qquad \lambda:=2\bigl(2c(\overline W,s)\bigr)^{-s}.
\]
Choose \(s\in(1/3,\frac12)\) close enough to \(\frac12\) and then \(\overline W\) small enough so that \(1<\lambda<2^s\). Pick \(A\) with \(\lambda^{1/s}<A<2\), and then choose \(\varepsilon>0\) so small that \(A2^{3\varepsilon}<2\). Setting
\[
\chi_1=A^n,\qquad \chi_2=2^{\varepsilon n},\qquad \chi_3=2^{-\varepsilon n},
\]
we obtain
\[
\mathbb P\left(\sum_{i\in \Lambda_n}e^{u_i^{(1)}}>\chi_1\right)\le C(\overline W,s)\bigl(\lambda/A^s\bigr)^n,
\qquad \mathbb P\bigl(e^{u_{\delta_n}^{(1)}}>\chi_2\bigr)\le 2^{-\varepsilon n},
\]
and
\[
\mathbb P\bigl(\pi_1^{(n)}<\chi_3\bigr)\le C_0e^{-c_0 2^{\varepsilon n}}.
\]
On the complement of these events,
\[
 P_1^{\widetilde\Lambda_n}(\tau_{\delta_n}<\tau_1^+)\le C\,2^{-n}A^n2^{3\varepsilon n},
\]
which is summable because \(A2^{3\varepsilon}<2\). The same Borel--Cantelli argument yields recurrence in the critical strong-reinforcement regime.
\subsection{Transience of VRJP when \(d>2\)}
\label{sec:org1910b0a}
By \cite[Theorem~3.2]{disertori23:_trans}, for every \(m\ge 1\),
\[
\sup_L \sup_{j\in \Lambda_L} \mathbb E_{\Lambda_L}\bigl(\cosh(u_j^{(L)})^m\bigr)\le c_H(m)<\infty.
\]
The only new point relative to \cite{disertori23:_trans} is to translate these field estimates into transience of the VRJP through the random-conductance representation.

Recall that \(C_{i,j}\) denotes the random conductances of the discrete-time chain associated with the VRJP. To prove transience, it suffices to estimate the effective resistance between \(1\) and \(\delta_L\) in the wired finite-volume network:
\[
R^{L}(C,1,\delta_{L}) = \inf_{(\theta_{e})} \sum_{e} \frac{\theta_{e}^2}{C_{e}},
\]
where the infimum is over all unit current flows \(\theta_e\) from \(1\) to \(\delta_L\); see \cite[Section~2.4, in particular Exercise~2.13]{Lyons2015}.

Let \(\overline\theta\) be a unit current flow realizing the effective resistance \(R^L(W,1,\delta_L)\) for the underlying conductance network \((W_{i,j})\). Since the conductance-weighted walk with generator \(L_W\) is transient when \(d>2\) \cite[Eq.~(3.1) and Theorem~3.1]{Kritchevski_2008}, the resistances \(R^L(W,1,\delta_L)\) remain bounded as \(L\to\infty\). On the other hand,
\[
R^{L}(C,1,\delta_{L}) \leq \sum_{e} \frac{\overline{\theta}_{e}^2}{W_{e}} \frac{W_{e}}{C_{e}}.
\]
The expected number of visits to \(1\) before hitting \(\delta_L\) is bounded by \(\mathbb{E}(C_{1} R^{L}(C,1,\delta_{L}))\), so it is enough to control this expectation.

Suppose first that there exists \(c>0\) such that for every edge \(e\) of \(\Lambda_L\),
\[
\sum_{f\ni 1} \mathbb E_{\Lambda_L}\left(\frac{C_f}{C_e}\right)\le \frac{c}{W_e}.
\]
Then, writing \(e=\{i,j\}\),
\begin{equation*}\begin{aligned}
\mathbb{E}(C_{1} R^{L}(C,1,\delta_{L})) &\leq \mathbb{E}( \sum_{k\sim 1}C_{1,k} \sum_{e} \frac{\overline{\theta}_{e}^2}{W_{e}} \frac{W_{e}}{C_{e}})\\
&  \leq \sum_{1\sim k} \sum_{i\sim j} \frac{\overline{\theta}_{i,j}^2}{W_{i,j}} \mathbb{E}\left(\frac{C_{1,k}}{C_{i,j}}\right) W_{e} \\
& \leq  c\sum_{i\sim j} \frac{\overline{\theta}_{i,j}^2}{W_{i,j}}.
\end{aligned}\end{equation*}
Therefore transience follows from the transience of the conductance model with weights \(W\). Since the finite-volume expected number of visits is uniformly bounded in \(L\), monotone convergence gives a finite infinite-volume expected number of visits, and hence the number of visits is finite almost surely.

It remains to verify the assumption. We have
\begin{equation*}\begin{aligned}
\mathbb{E}\left( \frac{C_{1,k}}{C_{i,j}} \right) & = \mathbb{E}\left( \frac{W_{1,k} e^{u_{1}+u_{k}}}{W_{i,j} e^{u_{i}+u_{j}}} \right)\\
&\leq \frac{W_{1,k}}{W_{i,j}} \mathbb{E}(e^{4u_{1}})^{\frac{1}{4}} \mathbb{E}(e^{4u_{k}})^{\frac{1}{4}}\mathbb{E}(e^{-4u_{i}})^{\frac{1}{4}} \mathbb{E}(e^{-4u_{j}})^{\frac{1}{4}}.
\end{aligned}\end{equation*}
The uniform moment bounds above show that the right-hand side is bounded by a constant times \(W_{1,k}/W_{i,j}\), and summing over \(k\sim 1\) gives the required estimate.

\section{Fractional moment estimates}
\label{h:6fd27b89-918f-4c37-84a1-e95db8e62cef}
We start with the proof of part~1 of Theorem \ref{thm-FMM-exp-u}. The argument follows the classical fractional-moment method from \cite{Aizenman1993}. We mainly use two tools. The first is that, for the \(H^{2|2}\) model on a finite graph with zero boundary; that is, with density given by \eqref{eq-beta-density-Lambdatilda-n}, the law of \((2G_{\beta,\widetilde{\Lambda}_{n}}(i,i))^{-1}\) is Gamma with shape \(\frac{1}{2}\) and rate \(1\) for every \(i\in \widetilde{\Lambda}_n\); this is summarized in Lemma~\ref{lem-gamma-dist-Gii}.

The second tool is an exact coarse-graining identity for the \(H^{2|2}\) model on the hierarchical lattice; see \cite[Corollary~2.3]{Disertori2022}. We restate the version needed here as Lemma~\ref{lem-effective-h22} and Corollary~\ref{coro-effective-model-u}. For convenience, we first introduce some notation:
\subsection{Estimate for \(G(i,\delta)\)}
\label{sec:orgda94508}
\begin{definition}
  (Coarse-Grained Graph Partitioning)
  Let \(\mathcal{G} = (V, E, W)\) be a finite weighted graph with vertex set \(V\), edge set \(E\), and weight function \(W: E \to \mathbb{R}^+\). Partition the vertex set \(V\) into disjoint subsets \(V = B_1 \sqcup B_2 \sqcup \cdots \sqcup B_k\), where each \(B_j\) forms a block of the partition. Construct a new graph \(\mathcal{G}' = (V', E', W')\) by mapping each subset \(B_j\) to a coarse-grained vertex \(B_j \in V'\) (this slightly abusive notation will entail no ambiguity). The weight \(W'_{B_i, B_j}\) of an edge \(\{B_{i},B_{j}\}\) in \(\mathcal{G}'\) is defined as:
\begin{equation}\begin{aligned}
\label{eq-block-W-general}
W'_{B_i, B_j} = \sum_{\substack{u \in B_i,\\ v \in B_j}} W_{u,v} 
\end{aligned}\end{equation}
The graph \(\mathcal{G}'\) is called the coarse-grained graph associated with the partition \(V=B_1\sqcup \cdots \sqcup B_k\).  
\end{definition}

We now apply this definition to  \(\widetilde{\Lambda}_{n}\) with some particular choices of partitions. Note that its vertex set is \(\{1, 2, \dots, 2^n, \delta_n\}\). Define \(\widetilde{\Lambda}_{n}^{(0)}\) as the graph obtained by identifying groups of vertices: \(B_0' = \{1\}\), \(B_0 = \{2\}\), \(B_1 = \{3, 4\}\), \(B_2 = \{5, 6, 7, 8\}\), and continuing with \(B_j = \{2^{j} + 1, \dots, 2^{j+1}\}\) for \(1 \leq j \leq n-1\). The resulting vertex set of \(\widetilde{\Lambda}_{n}^{(0)}\) is \(\{B_0', B_0, B_1, \dots, B_{n-1}, \delta_n\}\), and it forms a complete graph.

The edge weight \(W_{B_j, B_k}\) is defined as the sum of the edge weights between all pairs of vertices in \(B_j\) and \(B_k\), given by, for \(j,k<n\),
\begin{equation}\begin{aligned}
\label{eq-block-W-BjBk}
W_{B_j, B_k} = 2^{j+k} \overline{W} (2\rho)^{-(j \vee k) - 1}, \ W_{\delta_{n},B_{j}}=2^{j} \frac{\overline{W} \rho^{-n}}{2(\rho-1)}. 
\end{aligned}\end{equation}
Indeed, if \(j\le k\), then every pair \((u,v)\in B_j\times B_k\) has hierarchical distance \(d_{\mathbb X}(u,v)=k+1\), so each summand in \eqref{eq-block-W-general} equals \(\overline W(2\rho)^{-(k+1)}\). Since \(|B_j|=2^j\) and \(|B_k|=2^k\), there are \(2^{j+k}\) such pairs, which gives
\[
W_{B_j,B_k}=2^{j+k}\overline W(2\rho)^{-(k+1)}=2^{j+k}\overline W(2\rho)^{-(j\vee k)-1}.
\]
The case \(j>k\) is symmetric. For the boundary vertex, \eqref{eq-edgeweight} shows that \(W_{\delta_n,i}=\overline W\rho^{-n}/(2(\rho-1))\) for every \(i\in\Lambda_n\), hence
\[
W_{\delta_n,B_j}=\sum_{i\in B_j}W_{\delta_n,i}=2^j\frac{\overline W\rho^{-n}}{2(\rho-1)}.
\]
Moreover, the adjacent coarse-grained weights satisfy
\[
W_{B_j,B_{j+1}}=\frac{2}{\rho}W_{B_{j-1},B_j}.
\]
Since \(d=2\log 2/\log \rho\), we have \(d<2\iff \rho>2\) and \(d>2\iff \rho<2\). Thus successive coarse-graining decreases these effective edge weights when \(d<2\), and increases them when \(d>2\). See also the discussion in \cite[Section~3.2]{HuangWangZeng2025}.

For \(k =  1, \dots, n-1\), define the graph \(\widetilde{\Lambda}_n^{(k)}\) with vertex groups
\begin{equation}\begin{aligned}
\label{eq-BkBkprimeetc}
&B_{k'} = \{1, \dots, 2^k\}, \\
& B_{k} = \{2^k + 1, \dots, 2^{k+1}\},\\
& B_{k+1}=\{2^{k+1}+1, \dots,2^{k+2}\} \\
& \dots \\
& B_{n-1} = \{2^{n-1} + 1, \dots, 2^n\}. 
\end{aligned}\end{equation}
The vertex set of \(\widetilde{\Lambda}_n^{(k)}\) is \(\{B_{k'}, B_{k}, B_{k+1},\dots, B_{n-1}, \delta_n\}\), forming a complete graph. The edge weight \(W_{B_j, B_m}\) is inherited from \(\widetilde{\Lambda}_n^{(0)}\), or equivalently, obtained from summing edge weights between all pairs of vertices in \(B_j\) and \(B_k\), for \(j,k\in \widetilde{\Lambda}_{n}^{(k)}\).

This construction progressively coarsens the graph by increasing the minimum size of the vertex groups from \(2^0 = 1\) in \(\widetilde{\Lambda}_n^{(0)}\) to \(2^k\) in \(\widetilde{\Lambda}_n^{(k)}\).

\begin{figure}[ht]
  \centering
  \begin{tikzpicture}[
      level distance=1cm,          
      level 1/.style={sibling distance=15em},
      level 2/.style={sibling distance=7.5em},
      level 3/.style={sibling distance=3.75em},
      level 4/.style={sibling distance=1.5em},
      every node/.style={circle, draw, minimum size=5mm, inner sep=0pt}
    ]

    \node {\(B_{4'}\)}
    child { node {\(B_{3'}\)}
      child { node {\(B_{2'}\)}
        child { node {\(B_{1'}\)}
          child { node {1} }
          child { node {2} }
        }
        child { node {\(B_{1}\)}
          child { node {3} }
          child { node {4} }
        }
      }
      child { node {\(B_{2}\)}
        child { node {}
          child { node {5} }
          child { node {6} }
        }
        child { node {}
          child { node {7} }
          child { node {8} }
        }
      }
    }
    child { node {\(B_{3}\)}
      child { node {}
        child { node {}
          child { node {9} }
          child { node {10} }
        }
        child { node {}
          child { node {11} }
          child { node {12} }
        }
      }
      child { node {}
        child { node {}
          child { node {13} }
          child { node {14} }
        }
        child { node {}
          child { node {15} }
          child { node {16} }
        }
      }
    };

    \node at (5,0) {\(\delta_4\)};
  \end{tikzpicture}
\caption{The box \(\widetilde{\Lambda}_4\); in this coarse-grained graph we use the convention \(B_4=\delta_4\).}
  \label{fig:tikzpic}
\end{figure}

We now state an equivalent version of Corollary 2.3 of \cite{Disertori2022}:
\begin{lemma}
Given a graph \(\mathcal{G}=(V,E,W)\), consider the random Schrödinger matrix \(H_{\beta}\) on \(\mathcal{G}\), assume moreover that there are vertices \(B=\{i_1,\cdots,i_{k}\}\subset V\) such that the following holds:
\[W_{i,j}=W_{i,j'},\ \forall i\notin B,\forall j,j'\in B.\]
Fix \(\delta\in V \setminus B\), consider \(\mathcal{G}'\), the coarse-grained graph of \(\mathcal{G}\) by the partition \(V=B\sqcup \left(\bigsqcup_{j\notin B} \{j\}\right)\), let \(H'_{\beta}\) be the random Schrödinger matrix on \(\mathcal{G}'\).  Define \(e^{u^{(\delta)}_{i}} = \frac{G_{\beta}(\delta,i)}{G_{\beta}(\delta,\delta)}\), let \(e^{u^{(\delta)}_{B}} = \frac{1}{\operatorname{card}B} \sum_{j\in B} e^{u^{(\delta)}_{j}}\), also define \(e^{u^{(\delta)'}_{j}} = \frac{G'_{\beta}(\delta,j)}{G'_{\beta}(\delta,\delta)}\), then we have equality in distribution
\[\left\{u_{B}^{(\delta)},\ (u_{i}^{(\delta)})_{i\notin B}\right\} \overset{(\ell)}{=} \left\{u^{(\delta)'}_{B}, (u^{(\delta)'}_{i})_{i\notin B}\right\} .\]
\label{lem-effective-h22}
\end{lemma}
\begin{proof}[Proof of Lemma \ref{lem-effective-h22}]
The joint law of \((u^{\delta}_{i})_{i\in V}\) can be realized as the random environment of the VRJP \(Y_{t}\) on \(\mathcal{G}\) with starting point \(\delta\) and initial local time 1, see e.g. Theorem 2 in \cite{Sabot2015} or Theorem 3 in \cite{Sabot2017}. We now examine \(Y_{t}\) while conceptualizing \(B\) as a single vertex; that is, when \(Y_{t}\) is located in \(B\), we do not distinguish between the specific vertices within \(B\). This can be interpreted as the VRJP \(Y'\) on \(\mathcal{G}'\), also starting from \(\delta\). When the process \(Y_{t}\) is outside \(B\), say at vertex \(i\), the probability of jumping to \(B\) is identical in both \(\mathcal{G}\) and \(\mathcal{G}'\), the jump rate from \(i\) to \(B\) for \(Y_{t}\) is
\[\sum_{j\in B} W_{i,j} e^{u_j-u_i} = |B| W_{i,j} e^{-u_i} \left(\frac{\sum_{j\in B} e^{u_j}}{|B|}\right) = W_{i,B} e^{u_{B}-u_i}.\]
Even though the process will transition between different sites within \(B\), the total local time accumulated in \(B\) remains consistent for both processes, the identity in distribution hence follows.
\end{proof}

We can apply this lemma to our coarse-grained graphs \(\widetilde{\Lambda}_{n}^{(k)}\), and we have the following convenient corollary
\begin{corollary}
Consider the effective \(H^{2|2}\) field \((u^{(n)}_{i})\) on \(\widetilde{\Lambda}_{n}\) with pinning point \(\delta_{n}\). Recall that \(\widetilde{\Lambda}_{n}^{(k)}\) is the coarse-grained graph obtained from the partition \(\{B_{k'}, B_{k}, B_{k+1},\dots, B_{n-1}, \delta_n\}\). Let \((u^{(n)'}_{i})\) be the effective \(H^{2|2}\) field on \(\widetilde{\Lambda}_{n}^{(k)}\), and define \(u^{(n)}_{B_{j}}\) by
\[
e^{u^{(n)}_{B_{j}}} = \frac{1}{\operatorname{card}B_{j}} \sum_{x\in B_{j}} e^{u^{(n)}_{x}}.
\]
Then we have equality in distribution
\[
\bigl(u_{B_{j}}^{(n)}\bigr)_{j} \overset{(\ell)}{=} \bigl(u_{B_{j}}^{(n)'}\bigr)_{j}.
\]
\label{coro-effective-model-u}
\end{corollary}

Our strategy is to first prove the recursive bound of Proposition~\ref{prop-1-recursiveBdd-uBk}, which controls the \(s\)-moment of \(e^{u_{B_{k'}}}\) by moments at coarser scales, and then iterate that bound to obtain Proposition~\ref{prop-2-FMMbound-uBk}.

Before starting the proof, fix \(k\in\{0,\dots,n-1\}\) and consider the effective \(H^{2|2}\) field on \(\widetilde{\Lambda}_{n}^{(k)}\), whose vertices are
\[
B_{k'},B_k,B_{k+1},\dots,B_{n-1},\delta_n.
\]
We suppress the pinning superscript and write
\[
e^{u_i}=\frac{G_{\beta,\widetilde{\Lambda}_n^{(k)}}(\delta_n,i)}{G_{\beta,\widetilde{\Lambda}_n^{(k)}}(\delta_n,\delta_n)}.
\]
The relevant graph will always be clear from the expectation \(\mathbb E_{\widetilde\Lambda_n^{(k)}}(\cdot)\). We also choose a coupling such that there exists a random variable \(\gamma\), independent of all these effective fields, with
\[
G_{\beta,\widetilde{\Lambda}_n^{(k)}}(\delta_n,\delta_n)=\frac{1}{2\gamma}, \qquad k=0,\dots,n-1.
\]

By Corollary \ref{coro-effective-model-u}, we have for every bounded continuous function \(f\),
\[\mathbb{E}_{\widetilde{\Lambda}_{n}}(f(u_{1}))= \mathbb{E}_{\widetilde{\Lambda}_{n}^{(0)}}(f(u_1)).\]

The following proposition gives a recursive inequality that bounds the fractional moment at a given scale by a sum over fractional moments at a coarser scale. Iterating this inequality in a suitable way will yield the desired fractional moment estimates.
\begin{proposition}
For any \(n\) and \(k<n\), any \(0 < s < \frac{1}{2}\), we have
\begin{equation}\begin{aligned}
 \mathbb{E}_{\widetilde{\Lambda}_{n}^{(k)}}(e^{{s u_{B_{k'}}}})  
 \leq  \left(1+ c_s^{s} W_{B_{k'},B_{k}}^{s}\right) c_s^{s} \left(\sum_{\ell=k+1}^{n} W_{B_{k},B_{\ell}}^{s} \mathbb{E}_{\widetilde{\Lambda}_{n}^{(\ell)}}(e^{{s u_{B_{\ell}}}})\right).
\end{aligned}\end{equation}
\label{prop-1-recursiveBdd-uBk}
\end{proposition}
Here and below, when \(\ell=n\), the term inside \(\mathbb{E}_{\widetilde{\Lambda}_{n}^{(n)}}\) is interpreted with \(B_{n}=\delta_{n}\) and \(e^{u_{B_{n}}}=1\).

\begin{proposition}
For any \(n\) and \(k<n\), and any \(0 < s < \frac{1}{2}\), there is a constant \(C(\overline{W},d,s)\), depending only on \(\overline{W},d,s\), such that
\[\mathbb{E}_{\widetilde{\Lambda}_{n}^{(k)}}(e^{{s u_{B_{k}}}}) \leq C(\overline{W},d,s)  \rho^{- {s (n-k)}} \text{ when } d<2, \]
and
\[\mathbb{E}_{\widetilde{\Lambda}_{n}^{(k)}}(e^{{s u_{B_{k}}}}) \leq C(\overline{W},d,s)  c(\overline{W},s)^{- {s (n-k)}} \text{ when } d=2\]
where
\[c(\overline{W},s) = \frac{2}{\left( 1+ \left(1+   \left(\frac{c_s\overline{W}}{4} \right)^{s}\right) \left( \frac{c_s\overline{W}}{4}  \right)^{s}\right)^{\frac{1}{s}}}.\]
\label{prop-2-FMMbound-uBk}
\end{proposition}

\begin{proof}[Proof of Theorem \ref{thm-FMM-exp-u}]
It suffices to let \(k=0\) in Proposition~\ref{prop-2-FMMbound-uBk}.
\end{proof}

\begin{proof}[Proof of Proposition \ref{prop-1-recursiveBdd-uBk}]
Fix \(0 < s < \frac{1}{2}\) and \(k\in\{0,\dots,n-1\}\), and write \(B_n=\delta_n\). By \cite[Proposition~6]{Sabot2019},
\begin{equation}\begin{aligned}
\label{eq-A-1}
\mathbb{E}_{\widetilde{\Lambda}_{n}^{(k)}}(G(\delta_{n},\delta_{n})^{s}) \mathbb{E}_{\widetilde{\Lambda}_{n}^{(k)}}(e^{{ s u_{B_{k'}}}}) & = \mathbb{E}_{\widetilde{\Lambda}_{n}^{(k)}}(G(B_{k'},\delta_{n})^{s})\\
& = \mathbb{E}_{\widetilde{\Lambda}_{n}^{(k)}}\left(\left(\sum_{\sigma:B_{k'}\to \delta_{n}} \left( \frac{W_{\sigma}}{(2 \beta)_{\sigma}}\right)\right)^{s}\right)
\end{aligned}\end{equation}
where the sum is over the set \(B_{k'} \to \delta_{n}\) of all paths \(\sigma\) going from \(B_{k'}\) to \(\delta_{n}\) in the graph \(\widetilde{\Lambda}_{n}^{(k)}\), and \(W_{\sigma} = \prod_{j=0}^{\ell-1} W_{\sigma_{j},\sigma_{j+1}}\) where \(\sigma\) is of length \(\ell\); and \((2\beta)_{\sigma} = \prod_{j=0}^{\ell} (2 \beta)_{\sigma_{j}}\), and \(\beta\) is the random potential of the \(H^{2|2}\) model on \(\widetilde{\Lambda}_{n}^{(k)}\).

We split the path sum according to the first jump after the last visit to \(B_{k'}\). Either the path next jumps to \(B_k\), in which case we denote by \(B_{k'}\Rightarrow B_k\to \delta_n\) the set of such paths, or it jumps directly to some block \(B_\ell\) with \(\ell>k\), in which case we denote by \(B_{k'}\Rightarrow B_{>k}\to \delta_n\) the set of such paths. Thus
\begin{equation}\begin{aligned}
\label{eq-pathsum-2-parts}
\sum_{\sigma:B_{k'}\to \delta_{n}} \left( \frac{W_{\sigma}}{(2 \beta)_{\sigma}}\right)= \sum_{\sigma:B_{k'}\Rightarrow B_{k}\to \delta_{n}} \left( \frac{W_{\sigma}}{(2 \beta)_{\sigma}}\right)+ \sum_{\sigma:B_{k'}\Rightarrow B_{>k}\to \delta_{n}} \left( \frac{W_{\sigma}}{(2 \beta)_{\sigma}}\right).
\end{aligned}\end{equation}
The following property will be used frequently throughout our argument: if \(x_1,\cdots,x_{m} >0\) and \(0<\alpha<1\), then
\begin{equation}\begin{aligned}
\label{eq-FMM-power-bound}
(x_1 + \cdots +x_{m})^{\alpha} \leq x_1^{\alpha} + \cdots +x_{m}^{\alpha}, 
\end{aligned}\end{equation}
though we will not mention it each time.

To analyze the first sum in Eq. \eqref{eq-pathsum-2-parts}, corresponding to paths that pass through the adjacent block \(B_k\), we perform a standard path decomposition based on the last visit to \(B_k'\) and \(B_k\).

For the first class of paths, we decompose the path at the last visit to \(B_{k'}\) and at the last visit to \(B_k\). The initial segment contributes the factor \(G(B_{k'},B_{k'})\); after the last visit to \(B_{k'}\), the path jumps to \(B_k\); after the last visit to \(B_k\), it jumps to some \(B_\ell\) with \(\ell>k\); and the remaining segment then goes from \(B_\ell\) to \(\delta_n\) without visiting \(B_{k'}\) or \(B_k\). Formally,
\begin{equation*}\begin{aligned}
&\sum_{B_{k'} \Rightarrow B_{k} \to  \delta_{n} } \frac{W_{\sigma}}{(2\beta)_{\sigma}}\\
& = G(B_{k'},B_{k'}) W_{B_{k'},B_{k}} \sum_{\sigma':B_{k} \xrightarrow[]{\cancel{B_{k'}}} B_{k}} \frac{W_{\sigma'}}{(2\beta)_{\sigma'}} \sum_{\ell=k+1}^{n} W_{B_{k},B_{\ell}} \sum_{\sigma': B_{\ell} \xrightarrow[]{\cancel{B_{k'},B_{k}}} \delta_{n}} \frac{W_{\sigma'}}{(2\beta)_{\sigma'}}
\end{aligned}\end{equation*}
where \(\sigma':B_{k} \xrightarrow[]{\cancel{B_{k'}}} B_{k}\) means the set of paths going from \(B_{k}\) to \(B_{k}\) without using the vertex \(B_{k'}\); similarly, \(B_{\ell} \xrightarrow[]{\cancel{B_{k'},B_{k}}} \delta_{n}\) denotes the set of paths going from \(B_{\ell}\) to \(\delta_{n}\) without visiting \(\{B_{k'},B_{k}\}\).

Under the law \(\mathbb{E}_{\widetilde{\Lambda}_{n}^{(k)}}\), the factor \(G(B_{k'},B_{k'})\) depends only on the diagonal variable at \(B_{k'}\), whereas the remaining path sums are measurable with respect to the field away from \(B_{k'}\) after its last visit. By Lemma~\ref{lem-gamma-dist-Gii}, \(G(B_{k'},B_{k'})\) is independent of the remaining factor. Hence
\begin{equation*}\begin{aligned}
&\mathbb{E}_{\widetilde{\Lambda}_{n}^{(k)}}\left( \left(\sum_{B_{k'} \Rightarrow B_{k} \to  \delta_{n} } \frac{W_{\sigma}}{(2\beta)_{\sigma}}\right)^{s}\right) \\
& = c_s^{s} \mathbb{E}_{\widetilde{\Lambda}_{n}^{(k)}}\left( \left(W_{B_{k'},B_{k}} \sum_{\sigma':B_{k} \xrightarrow[]{\cancel{B_{k'}}} B_{k}} \frac{W_{\sigma'}}{(2\beta)_{\sigma'}} \sum_{\ell=k+1}^{n} W_{B_{k},B_{\ell}} \sum_{\sigma': B_{\ell} \xrightarrow[]{\cancel{B_{k'},B_{k}}} \delta_{n}} \frac{W_{\sigma'}}{(2\beta)_{\sigma'}}\right)^{s}\right)
\end{aligned}\end{equation*}

Next
\begin{equation*}\begin{aligned}
& \sum_{\sigma':B_{k} \xrightarrow[]{\cancel{B_{k'}}} B_{k}} \frac{W_{\sigma'}}{(2\beta)_{\sigma'}} \sum_{\ell=k+1}^{n} W_{B_{k},B_{\ell}} \sum_{\sigma': B_{\ell} \xrightarrow[]{\cancel{B_{k'},B_{k}}} \delta_{n}} \frac{W_{\sigma'}}{(2\beta)_{\sigma'}} \\
&\leq   \sum_{\sigma':B_{k} \xrightarrow[]{} B_{k}} \frac{W_{\sigma'}}{(2\beta)_{\sigma'}} \sum_{\ell=k+1}^{n} W_{B_{k},B_{\ell}} \sum_{\sigma': B_{\ell} \xrightarrow[]{\cancel{B_{k'},B_{k}}} \delta_{n}} \frac{W_{\sigma'}}{(2\beta)_{\sigma'}}\\
& =G(B_{k},B_{k}) \sum_{\ell=k+1}^{n} W_{B_{k},B_{\ell}} \sum_{\sigma': B_{\ell} \xrightarrow[]{\cancel{B_{k'},B_{k}}} \delta_{n}} \frac{W_{\sigma'}}{(2\beta)_{\sigma'}}
\end{aligned}\end{equation*}
and again \(G(B_{k},B_{k})\) depends only on the diagonal variable at \(B_k\), while the remaining factor is measurable away from \(B_k\). By Lemma~\ref{lem-gamma-dist-Gii}, it is independent of that factor, so
\begin{equation*}\begin{aligned}
&\mathbb{E}_{\widetilde{\Lambda}_{n}^{(k)}}\left( \left(\sum_{\sigma':B_{k} \xrightarrow[]{\cancel{B_{k'}}} B_{k}} \frac{W_{\sigma'}}{(2\beta)_{\sigma'}} \sum_{\ell=k+1}^{n} W_{B_{k},B_{\ell}} \sum_{\sigma': B_{\ell} \xrightarrow[]{\cancel{B_{k'},B_{k}}} \delta_{n}} \frac{W_{\sigma'}}{(2\beta)_{\sigma'}} \right)^{s} \right) \\
& \leq c_s^{s} \sum_{\ell=k+1}^{n} W_{B_{k},B_{\ell}}^{s} \mathbb{E}_{\widetilde{\Lambda}_{n}^{(k)}}\left(  \left(  \sum_{\sigma': B_{\ell} \xrightarrow[]{\cancel{B_{k'},B_{k}}} \delta_{n}} \frac{W_{\sigma'}}{(2\beta)_{\sigma'}}\right)^{s} \right)\\
& \leq c_s^{s} \sum_{\ell=k+1}^{n} W_{B_{k},B_{\ell}}^{s} \mathbb{E}_{\widetilde{\Lambda}_{n}^{(k)}}\left(  \left(  \sum_{\sigma': B_{\ell} \xrightarrow[]{} \delta_{n}} \frac{W_{\sigma'}}{(2\beta)_{\sigma'}}\right)^{s} \right)
\end{aligned}\end{equation*}
After dropping the constraint to avoid \(B_{k'},B_{k}\), the remaining factor is exactly
\begin{equation}\begin{aligned}
\label{eq-Bl-deltan-Geu}
\mathbb{E}_{\widetilde{\Lambda}_{n}^{(k)}}\left( \left(\sum_{\sigma':B_{\ell} \xrightarrow[]{} \delta_{n}} \frac{W_{\sigma'}}{(2 \beta)_{\sigma'}}\right)^{s}
\right) & = \mathbb{E}_{\widetilde{\Lambda}_{n}^{(k)}}\left(G(\delta_{n},B_{\ell})^{s}\right)\\
& = \mathbb{E}_{\widetilde{\Lambda}_{n}^{(k)}}(G(\delta_{n},\delta_{n})^{s}) \, \mathbb{E}_{\widetilde{\Lambda}_{n}^{(k)}}\left(e^{s u_{B_{\ell}}}\right) \\
& = \mathbb{E}_{\widetilde{\Lambda}_{n}^{(k)}}(G(\delta_{n},\delta_{n})^{s}) \, \mathbb{E}_{\widetilde{\Lambda}_{n}^{(\ell)}}\left(e^{s u_{B_{\ell}}}\right).
\end{aligned}\end{equation}
where in the last equality we successively coarse-grain \(B_{k'}\) and \(B_k\), then \(B_{(k+1)'}\), and so on, until the block containing the starting point becomes \(B_{\ell'}\). This is exactly Corollary~\ref{coro-effective-model-u}. To summarize, our discussion of the first sum gives the following upper bound:
\begin{equation}\begin{aligned}
\label{eq-3A}
&\mathbb{E}_{\widetilde{\Lambda}_{n}^{(k)}}\left( \left(\sum_{B_{k'} \Rightarrow B_{k} \to  \delta_{n} } \frac{W_{\sigma}}{(2\beta)_{\sigma}}\right)^{s}\right) \\
&\leq c_s^{{2s}} W_{B_{k'},B_{k}}^{s} \sum_{\ell=k+1}^{n} W_{B_k,B_{\ell}}^{s} \mathbb{E}_{\widetilde{\Lambda}_{n}^{(k)}}(G(\delta_{n},\delta_{n})^{s}) \mathbb{E}_{\widetilde{\Lambda}_{n}^{(\ell)}}\left(e^{{s u_{B_{\ell}}}}\right) 
\end{aligned}\end{equation}

For the second class of paths, we factorize the initial segment up to the last visit to \(B_{k'}\), which contributes the factor \(G(B_{k'},B_{k'})\). Thus
\begin{equation*}\begin{aligned}
\sum_{B_{k'} \Rightarrow B_{>k} \to \delta_{n}} \frac{W_{\sigma}}{(2\beta)_{\sigma}} = G(B_{k'},B_{k'}) \sum_{\ell=k+1}^{n} W_{B_{k'},B_{\ell}} \sum_{\sigma':B_{\ell} \xrightarrow[]{\cancel{B_{k'}}} \delta_{n}}  \frac{W_{\sigma'}}{(2 \beta)_{\sigma'}}
\end{aligned}\end{equation*}
Under the law \(\mathbb{E}_{\widetilde{\Lambda}_{n}^{(k)}}\), \(G(B_{k'},B_{k'})\) is independent of \(\sum_{\sigma':B_{\ell} \xrightarrow[]{\cancel{B_{k'}}} \delta_{n}}  \frac{W_{\sigma'}}{(2 \beta)_{\sigma'}}\)  and it equals in law to \(\frac{1}{2\gamma}\), therefore,
\begin{equation*}\begin{aligned}
&\mathbb{E}_{\widetilde{\Lambda}_{n}^{(k)}}\left(\left(\sum_{B_{k'} \Rightarrow B_{>k} \to \delta_{n}} \frac{W_{\sigma}}{(2\beta)_{\sigma}}\right)^{s}\right) \\
& \le \sum_{\ell=k+1}^{n} W_{B_{k'},B_{\ell}}^{s}  \mathbb{E}_{\widetilde{\Lambda}_{n}^{(k)}}\left(G(B_{k'},B_{k'})^{s}\right) \mathbb{E}_{\widetilde{\Lambda}_{n}^{(k)}}\left( \left(\sum_{\sigma':B_{\ell} \xrightarrow[]{\cancel{B_{k'}}} \delta_{n}}  \frac{W_{\sigma'}}{(2 \beta)_{\sigma'}}\right)^{s}
\right)
\\ &\leq  \sum_{\ell=k+1}^{n} W_{B_{k'},B_{\ell}}^{s}  \mathbb{E}_{\widetilde{\Lambda}_{n}^{(k)}}\left(G(B_{k'},B_{k'})^{s}\right) \mathbb{E}_{\widetilde{\Lambda}_{n}^{(k)}}\left( \left(\sum_{\sigma':B_{\ell} \xrightarrow[]{} \delta_{n}}  \frac{W_{\sigma'}}{(2 \beta)_{\sigma'}}\right)^{s}
\right)\\
& =  \sum_{\ell=k+1}^{n} W_{B_{k'},B_{\ell}}^{s}  c_s^{s}\mathbb{E}_{\widetilde{\Lambda}_{n}^{(k)}}\left( \left(\sum_{\sigma':B_{\ell} \xrightarrow[]{} \delta_{n}}  \frac{W_{\sigma'}}{(2 \beta)_{\sigma'}}\right)^{s}
\right)
\end{aligned}\end{equation*}
The last factor is already discussed in Eq. \eqref{eq-Bl-deltan-Geu}. To summarize, the second sum satisfies the following upper bound:
\begin{equation}\begin{aligned}
\label{eq-4A}
\mathbb{E}_{\widetilde{\Lambda}_{n}^{(k)}}\left(\left(\sum_{B_{k'} \Rightarrow B_{>k} \to \delta_{n}} \frac{W_{\sigma}}{(2\beta)_{\sigma}}\right)^{s}\right) \leq  \sum_{\ell=k+1}^{n} W_{B_{k'},B_{\ell}}^{s}  c_s^{s} \mathbb{E}_{\widetilde{\Lambda}_{n}^{(k)}}(G(\delta_{n},\delta_{n})^{s}) \mathbb{E}_{\widetilde{\Lambda}_{n}^{(\ell)}}\left(e^{{s u_{B_{\ell}}}}\right).
\end{aligned}\end{equation}

Recall that
\[\mathbb{E}_{\widetilde{\Lambda}_{n}^{(k)}}(G(\delta_{n},\delta_{n})^{s}) = c_{s}^{s} \in (0,\infty).\]

Therefore, combining Equations \eqref{eq-A-1}, \eqref{eq-pathsum-2-parts}, \eqref{eq-3A}, \eqref{eq-4A}, and dividing both sides by \(\mathbb{E}_{\widetilde{\Lambda}_{n}^{(k)}}(G(\delta_{n},\delta_{n})^{s})=c_s^s\), we obtain
\begin{equation}\begin{aligned}
\label{eq-recursion-fmm}
& \mathbb{E}_{\widetilde{\Lambda}_{n}^{(k)}}(e^{{s u_{B_{k'}}}})  \\
& \leq  c_s^{{2s}} W_{B_{k'},B_{k}}^{s} \sum_{\ell=k+1}^{n} W_{B_k,B_{\ell}}^{s} \mathbb{E}_{\widetilde{\Lambda}_{n}^{(\ell)}}\left(e^{{s u_{B_{\ell}}}}\right)  \\
& + \sum_{\ell=k+1}^{n} W_{B_{k'},B_{\ell}}^{s}  c_s^{s} \mathbb{E}_{\widetilde{\Lambda}_{n}^{(\ell)}}\left(e^{{s u_{B_{\ell}}}}\right) \\
& = \left(1+ c_s^{s} W_{B_{k'},B_{k}}^{s}\right) c_s^{s} \left(\sum_{\ell=k+1}^{n} W_{B_{k},B_{\ell}}^{s} \mathbb{E}_{\widetilde{\Lambda}_{n}^{(\ell)}}(e^{{s u_{B_{\ell}}}})\right).
\end{aligned}\end{equation}
\end{proof}

\begin{proof}[Proof of Proposition \ref{prop-2-FMMbound-uBk}]
Note that \(\mathbb{E}_{\widetilde{\Lambda}_{n}^{(k)}}(e^{{s u_{B_{k}}}})= \mathbb{E}_{\widetilde{\Lambda}_{n}^{(k)}}(e^{{s u_{B_{k'}}}})\).

We can directly iterate this bound by writing
\begin{equation*}\begin{aligned}
&\mathbb{E}_{\widetilde{\Lambda}_{n}^{(k)}}(e^{{s u_{B_{k}}}})  \\
& \leq  \left(1+ c_s^{s} W_{B_{k'},B_{k}}^{s}\right) c_s^{s} \left(\sum_{k_1=k+1}^{n} W_{B_{k},B_{k_1}}^{s} \mathbb{E}_{\widetilde{\Lambda}_{n}^{(k_1)}}(e^{{s u_{B_{k_1}}}})\right) \\
& \leq  \left(1+ c_s^{s} W_{B_{k'},B_{k}}^{s}\right) c_s^{s} \left(\sum_{k_1=k+1}^{n} W_{B_{k},B_{k_1}}^{s} \left(1+ c_s^{s} W_{B_{k_1'},B_{k_1}}^{s}\right) c_s^{s}\left( \sum_{k_2=k_1+1}^{n} W_{B_{k_1},B_{k_2}}^{s} \mathbb{E}_{\widetilde{\Lambda}_{n}^{(k_2)}} (e^{ {s u_{B_{k_2}}}})\right) \right) \\
& \vdots \\
& \leq   \sum_{k=k_0<k_1<k_2 <\cdots < k_{\ell}=n} \prod_{i=0}^{\ell-1} \left[\left(1+ c_s^{s} W_{B_{k_i'},B_{k_i}}^{s}\right)c_s^{s} W_{B_{k_i},B_{k_{i+1}}}^{s}\right]
\end{aligned}\end{equation*}
Substituting the explicit value of \(W_{B_{k_i},B_{k_{i+1}}}\), we continue our computation
\begin{equation*}\begin{aligned}
& = \sum_{k=k_0<k_1<k_2 <\cdots < k_{\ell}=n} \prod_{i=0}^{\ell-1} \left[\left(1+ \left( c_s2^{2 k_i} \overline{W} (2\rho)^{-k_i-1}\right)^{s}\right) c_s^{s} \left( \frac{2^{k_i+k_{i+1}}\overline{W}}{(2\rho)^{k_{i+1}+1}} (\frac{\rho}{\rho-1}\mathbbm{1}_{i=\ell-1} + \mathbbm{1}_{i<\ell-1})\right)^{s}  \right]\\
& =  \sum_{k=k_0<k_1<k_2 <\cdots < k_{\ell}=n}  \left( \frac{\rho}{\rho-1}\right)^{s} \prod_{i=0}^{\ell-1}\left[ \left(1+   \left(\frac{c_s\overline{W}}{2\rho} \left(\frac{2}{\rho}\right)^{k_i} \right)^{s}\right) \left( \frac{c_s\overline{W}}{2\rho}  \frac{2^{k_i}}{\rho^{k_{i+1}}}\right)^{s}\right]\\
& = \left( \frac{\rho}{\rho-1}\right)^{s}  \sum_{k=k_0<k_1<k_2 <\cdots < k_{\ell}=n} \prod_{i=0}^{\ell-1} \left[ \left(1+   \left(\frac{c_s\overline{W}}{2\rho} \left(\frac{2}{\rho}\right)^{k_i} \right)^{s}\right) \left( \frac{c_s\overline{W}}{2\rho}  \frac{2^{k_i}}{\rho^{k_{i}}}\right)^{s}\right] \prod_{i=0}^{\ell-1} \left(\frac{\rho^{k_i}}{\rho^{k_{i+1}}}\right)^{s} \\
& = \left( \frac{\rho}{\rho-1}\right)^{s} \rho^{-{s (n-k)}} \sum_{k=k_0<k_1<k_2 <\cdots < k_{\ell}=n} \prod_{i=0}^{\ell-1} \left[ \left(1+   \left(\frac{c_s\overline{W}}{2\rho} \left(\frac{2}{\rho}\right)^{k_i} \right)^{s}\right) \left( \frac{c_s\overline{W}}{2\rho}  \frac{2^{k_i}}{\rho^{k_{i}}}\right)^{s}\right] 
\end{aligned}\end{equation*}
Since \(\sum_{k=k_0<k_1<k_2 <\cdots < k_{\ell}=n} \prod_{i=0}^{\ell-1} A_{k_i} \le \prod_{i=k}^{n-1} (1+A_i)\), we have, to summarize for now,
\begin{equation}\begin{aligned}
\label{eq-ELambda-euBk-bound-with-rho}
&\mathbb{E}_{\widetilde{\Lambda}_{n}^{(k)}}(e^{{s u_{B_{k}}}}) \\
& \leq  \left( \frac{\rho}{\rho-1}\right)^{s} \rho^{-{s (n-k)}}  \prod_{i=k}^{n-1} \left[1+ \left(1+   \left(\frac{c_s\overline{W}}{2\rho} \left(\frac{2}{\rho}\right)^{i} \right)^{s}\right) \left( \frac{c_s\overline{W}}{2\rho}  \frac{2^{i}}{\rho^{i}}\right)^{s}\right]
\end{aligned}\end{equation}
When \(d<2\), the product above is bounded uniformly in \(n\), for all \(\overline{W}\):
\begin{equation*}\begin{aligned}
&\prod_{i=k}^{n-1} \left[1+ \left(1+   \left(\frac{c_s\overline{W}}{2\rho} \left(\frac{2}{\rho}\right)^{i} \right)^{s}\right) \left( \frac{c_s\overline{W}}{2\rho}  \frac{2^{i}}{\rho^{i}}\right)^{s}\right] \\
&  \leq \exp \sum_{i=0}^{\infty}  \left(1+   \left(\frac{c_s\overline{W}}{2\rho} \left(\frac{2}{\rho}\right)^{i} \right)^{s}\right) \left( \frac{c_s\overline{W}}{2\rho}  \frac{2^{i}}{\rho^{i}}\right)^{s}\\
& <\infty
\end{aligned}\end{equation*}
For \(d<2\) (that is, \(\rho>2\)), the term \((2/\rho)^s\) is less than 1, ensuring the infinite product converges and is bounded uniformly in \(n\).

Therefore, there exists a constant \(C(\overline{W},d,s)\) such that, if \(d<2\),
\begin{equation}\begin{aligned}
\label{eq-sum-one-path-1plusAj}
\sum_{k=k_0<k_1<k_2 <\cdots < k_{\ell}=n} \prod_{i=0}^{\ell-1} \left[\left(1+ c_s^{s} W_{B_{k_i'},B_{k_i}}^{s}\right)c_s^{s} W_{B_{k_i},B_{k_{i+1}}}^{s}\right] \leq  C(\overline{W},d,s)  \rho^{- {s (n-k)}}
\end{aligned}\end{equation}
As a consequence:
\[\mathbb{E}_{\widetilde{\Lambda}_{n}^{(k)}}(e^{{s u_{B_{k}}}}) \leq C(\overline{W},d,s)  \rho^{- {s (n-k)}}.\]

When \(d=2\), then \(\rho=2\), we choose \(\overline{W}\) small enough such that
\begin{equation}\begin{aligned}
\label{eq-choose-upsilon}
1+ \left(1+   \left(\frac{c_s\overline{W}}{2\rho} \left(\frac{2}{\rho}\right)^{i} \right)^{s}\right) \left( \frac{c_s\overline{W}}{2\rho}  \frac{2^{i}}{\rho^{i}}\right)^{s} = 1+ \left(1+   \left(\frac{c_s\overline{W}}{4} \right)^{s}\right) \left( \frac{c_s\overline{W}}{4}  \right)^{s} < 2^{s}.
\end{aligned}\end{equation}
Therefore, we have
\begin{equation}\begin{aligned}
\label{eq-sum-one-path-1plusAj-d-is-2}
&\sum_{k=k_0<k_1<k_2 <\cdots < k_{\ell}=n} \prod_{i=0}^{\ell-1} \left[\left(1+ c_s^{s} W_{B_{k_i'},B_{k_i}}^{s}\right)c_s^{s} W_{B_{k_i},B_{k_{i+1}}}^{s}\right] \\
&\leq   \left( \frac{\rho}{\rho-1}\right)^{s} \rho^{-{s (n-k)}} \prod_{i=k}^{n-1} \left[1+ \left(1+   \left(\frac{c_s\overline{W}}{2\rho} \left(\frac{2}{\rho}\right)^{i} \right)^{s}\right) \left( \frac{c_s\overline{W}}{2\rho}  \frac{2^{i}}{\rho^{i}}\right)^{s}\right] \\
&\leq    2^{s} 2^{-{s (n-k)}}  \left[1+ \left(1+   \left(\frac{c_s\overline{W}}{4} \right)^{s}\right) \left( \frac{c_s\overline{W}}{4}  \right)^{s}\right]^{n-k}
\end{aligned}\end{equation}
Therefore, the stated estimate holds a fortiori with the smaller constant \(c(\overline{W},s)\) defined in Theorem \ref{thm-FMM-exp-u}.
\end{proof}

\subsection{Estimate for \(G(i,j)\)}
\label{sec:org69b702a}
The proof of the fractional-moment estimate for \(G(i,j)\) follows the same general strategy as the proof of the \(G(i,\delta_n)\) estimate in Section~\ref{sec:orgda94508}. The new feature is that the coarse-grain reduction step is more involved, because the path expansion now has to keep track of two distinguished vertices.

To estimate fractional moments of \(G_{\beta,\widetilde\Lambda_n}(1,j)\) for \(2\le j\le 2^n\), we coarse-grain the graph \(\widetilde\Lambda_n\) as much as possible while keeping the vertices \(1\) and \(j\) separate. For instance, in Figure~\ref{fig:tikzpic5} we merge every dyadic block that contains neither \(1\) nor \(j=9\).
\begin{figure}[ht]
  \centering
  \begin{tikzpicture}[
      level distance=1cm,
      level 1/.style={sibling distance=20em},
      level 2/.style={sibling distance=10em},
      level 3/.style={sibling distance=5em},
      level 4/.style={sibling distance=2.5em},
      level 5/.style={sibling distance=1.25em},
      every node/.style={circle, draw, minimum size=5mm, inner sep=0pt}
    ]
    \node {\(B_{5'}\)}
    child { node {\(B_{4'}\)}
      child { node {\(B_{3'}\)}
        child { node {\(B_{2'}\)}
          child { node {\(B_{1'}\)}
            child { node {1} }
            child { node {2} }
          }
          child { node {\(B_{1}\)}
          }
        }
        child { node {\(B_{2}\)}
        }
      }
      child { node {\(B_{3}\)}
        child { node {\(C_{2'}\)}
          child { node {\(C_{1'}\)}
            child { node {9} }
            child { node {10} }
          }
          child { node {\(C_1\)}
          }
        }
        child { node {\(C_2\)}
        }
      }
    }
    child { node {\(B_{4}\)}      
    };

    \node at (6,0) {\(\delta_5\)};
  \end{tikzpicture}
\caption{The graph \(\widetilde{\Lambda}_{5,\{1,9\}}^{(0)}\) obtained after coarse-graining the box \(\widetilde{\Lambda}_5\) in order to compute \(G_{\widetilde{\Lambda}_5}(1,9)\). Here \(\delta_5\) is also \(B_5\).}
  \label{fig:tikzpic5}
\end{figure}

The blocks \(B_k\) and \(B_{k'}\) are defined as in \eqref{eq-BkBkprimeetc}. To describe the corresponding blocks on the branch containing \(j\), assume without loss of generality that \(j=2^m+1\); see Figure~\ref{fig:tikzpic5}. Inside the subtree rooted at the right child of the youngest common ancestor of \(1\) and \(j\), we define
\[\begin{aligned}
&C_{1'} = \{2^{m}+1,2^{m}+2\}=2^{m}+B_{1'}\\
&C_{1} = \{2^{m}+3,2^{m}+4\}=2^{m}+B_{1}\\
&C_{2} = 2^{m}+B_{2}\\
&\vdots\\
&C_{m-1}=2^{m}+B_{m-1}.
\end{aligned}\]
Given \(1\) and \(j\), we define a sequence of increasingly coarse graphs \(\widetilde{\Lambda}_{n,\{1,j\}}^{(0)},\widetilde{\Lambda}_{n,\{1,j\}}^{(1)},\dots,\widetilde{\Lambda}_{n,\{1,j\}}^{(m)}\). The first graph \(\widetilde{\Lambda}_{n,\{1,j\}}^{(0)}\) is obtained by coarse-graining all dyadic blocks \(B_1,\dots,B_{m-1},C_1,\dots,C_{m-1}\) together with the blocks \(B_{m+1},\dots,B_{n-1}\); see Figure~\ref{fig:tikzpic5} for an example. Thus \(\widetilde{\Lambda}_{n,\{1,j\}}^{(0)}\) has vertex set \(\{B_{0'},B_{0},B_1,\dots,B_{m-1},B_{m+1},\dots,B_n,C_0,C_{0'},C_1,\dots,C_{m-1}\}\), with the convention \(B_n=\delta_n\), where \(C_{0'}=\{j\}\) and \(C_0=\{j+1\}\). In general, the vertex set of \(\widetilde{\Lambda}_{n,\{1,j\}}^{(k)}\) is
\begin{equation}\begin{aligned}
\label{eq-vertex-set-Lambdan1jk}
 \{B_{0'},B_{0},B_1 ,\cdots,B_{m-1},B_{m+1},\cdots,B_{n}, C_{k'},C_k,C_{k+1} ,\cdots, C_{m-1}\}. 
\end{aligned}\end{equation}

By Eq. \eqref{eq-block-W-general} we can compute the effective \(W\) between the \(B\)s and the \(C\)s: given \(0\leq k\leq m-1\), for \(\ell\geq k+1\) and \(i\in \{0 ,\cdots,m-1,m+1 ,\cdots,n\}\), with \(B_n=\delta_n\),
\begin{equation}\begin{aligned}
\label{eq-block-W-CkClBi}
&W_{C_{k},C_{\ell}}= W_{C_{k'},C_{\ell}}=W_{B_{k'},B_{\ell}}= W_{B_{k},B_{\ell}}=2^{k+\ell} \overline{W} (2\rho)^{-(k\vee \ell)-1}\\
& W_{C_{k'},B_{i}}=W_{C_k,B_i}=\begin{cases}
2^{k+i} \overline{W}(2\rho)^{-m-1}, & 0\le i\le m-1,\\
2^{k+i}\overline{W} (2\rho)^{-(k\vee i)-1}, & m+1\le i\le n-1,\\
2^{k} \dfrac{\overline{W}\rho^{-n}}{2(\rho-1)}, & i=n,
\end{cases}\\
& W_{C_{k'},C_{k}}= W_{B_{k'},B_{k}}= 2^{2k} \overline{W} (2\rho)^{-k-1}.
\end{aligned}\end{equation}

We first record the auxiliary estimate needed below.
\begin{proposition}
On \(\widetilde{\Lambda}_{n}^{(0)}\), for \(0\leq \ell\leq n\), we have the following estimate for the fractional moments of the effective field:
\begin{enumerate}
\item If \(d<2\), then there is a constant \(C''(\overline{W},d,s) > 0\) for all \(0\leq \ell\leq n\),
\[\mathbb{E}_{\widetilde{\Lambda}_{n}^{(0)}}(e^{{s u^{[1]}_{B_{\ell}}}}) \leq  C''(\overline{W},d,s) \rho^{-{\ell s}}.\]
\item If \(d=2\), then there is a constant \(C''(\overline{W},s)>0\) such that for all \(0\leq \ell\leq n\),
\[\mathbb{E}_{\widetilde{\Lambda}_{n}^{(0)}}(e^{{s u^{[1]}_{B_{\ell}}}}) \leq  C''(\overline{W},s) c(\overline{W},s)^{-{\ell s}}\]
where \(c(\overline{W},s)\) is the same as in Theorem \ref{thm-FMM-exp-u}.
\end{enumerate}
\label{prop-G-1-B-m}
\end{proposition}

\begin{proof}[Proof of Theorem \ref{thm-FMM-exp-uij}]
By Lemma~\ref{lem-effective-h22}, and recalling that \(j=2^m+1\),
\[\mathbb{E}_{\widetilde{\Lambda}_{n}}(e^{{s u^{[j]}_{1}}}) = \mathbb{E}_{\widetilde{\Lambda}_{n,\{1,j\}}^{(0)}}(e^{{s u^{[j]}_{1}}}).\]
Now fix \(k\in\{0,\dots,m-1\}\). Recall that \(C_{k'}=2^m+B_{k'}\). Using the vertex description in \eqref{eq-vertex-set-Lambdan1jk}, we obtain
\[\begin{aligned}
\mathbb{E}_{\widetilde{\Lambda}_{n,\{1,j\}}^{(k)}}(G(1,1)^{s})\mathbb{E}_{\widetilde{\Lambda}_{n,\{1,j\}}^{(k)}}(e^{{s u^{[1]}_{C_{k'}}}}) & = \mathbb{E}_{\widetilde{\Lambda}_{n,\{1,j\}}^{(k)}}(G(1,C_{k'})^{s})\\
& = \mathbb{E}_{\widetilde{\Lambda}_{n,\{1,j\}}^{(k)}} \left(\left(\sum_{\sigma:C_{k'}\to 1} \left(\frac{W_{\sigma}}{(2\beta)_{\sigma}}\right)\right)^{s}\right).
\end{aligned}\]
We first split the path at the last visit to \(C_{k'}\). After that last visit, the path either jumps to some \(B\)-block, or to one of the blocks \(C_{k+1},\dots,C_{m-1}\), or to \(C_k\). Accordingly,
\[\begin{aligned}
\sum_{\sigma:C_{k'}\to 1} \frac{W_{\sigma}}{(2\beta)_{\sigma}} & = G(C_{k'},C_{k'}) \sum_{i\in \{0 ,\cdots,m-1,m+1 ,\cdots,n\}}W_{C_{k'},B_{i}} \sum_{\sigma': B_i \xrightarrow[]{\cancel{C_{k'}}}1} \frac{W_{\sigma'}}{(2\beta)_{\sigma'}} \\
&+ G(C_{k'},C_{k'})\sum_{\ell=k+1}^{m-1} W_{C_{k'},C_{\ell}} \sum_{\sigma':C_{\ell} \xrightarrow[]{\cancel{C_{k'}} }1} \frac{W_{\sigma'}}{(2\beta)_{\sigma'}}\\
&+ G(C_{k'},C_{k'}) W_{C_{k'},C_{k}} \sum_{\sigma':C_{k}  \xrightarrow[]{\cancel{C_{k'}} }1} \frac{W_{\sigma'}}{(2\beta)_{\sigma'}}
\end{aligned}\]
Now
\[\begin{aligned}
& \mathbb{E}_{\widetilde{\Lambda}_{n,\{1,j\}}^{(k)}} \left(\left(\sum_{\sigma:C_{k'}\to 1} \left(\frac{W_{\sigma}}{(2\beta)_{\sigma}}\right)\right)^{s}\right) \\
& \leq  c_s^{s}   \sum_{i\in \{0 ,\cdots,m-1,m+1 ,\cdots,n\}}W_{C_{k'},B_{i}}^{s} \mathbb{E}_{\widetilde{\Lambda}_{n,\{1,j\}}^{(k)}}\left(\left(\sum_{\sigma': B_i \xrightarrow[]{\cancel{C_{k'}}}1} \frac{W_{\sigma'}}{(2\beta)_{\sigma'}}\right)^{s}\right) \\
&+ c_s^{s}  \sum_{\ell=k+1}^{m-1} W_{C_{k'},C_{\ell}}^{s} \mathbb{E}_{\widetilde{\Lambda}_{n,\{1,j\}}^{(k)}}\left(\left(\sum_{\sigma':C_{\ell} \xrightarrow[]{\cancel{C_{k'}} }1} \frac{W_{\sigma'}}{(2\beta)_{\sigma'}}\right)^{s}\right)\\
&+ c_s^{s}  W_{C_{k'},C_{k}}^{s} \mathbb{E}_{\widetilde{\Lambda}_{n,\{1,j\}}^{(k)}}\left(\left( \sum_{\sigma':C_{k}  \xrightarrow[]{\cancel{C_{k'}} }1} \frac{W_{\sigma'}}{(2\beta)_{\sigma'}}\right)^{s}\right)
\end{aligned}\]
The first term is controlled by Proposition~\ref{prop-G-1-B-m}. For the last term, we perform one further decomposition through the block \(C_k\). This gives the bound
\[\begin{aligned}
  &\sum_{\sigma':C_{k}  \xrightarrow[]{\cancel{C_{k'}} }1} \frac{W_{\sigma'}}{(2\beta)_{\sigma'}} \le \left(\sum_{\sigma':C_{k} \xrightarrow[]{\cancel{C_{k'}}} C_{k} } \frac{W_{\sigma'}}{(2\beta)_{\sigma'}}\right)\\
  & \times
\left( \sum_{i\in \{0 ,\cdots,m-1,m+1 ,\cdots,n\}}W_{C_{k},B_{i}} \sum_{\sigma': B_i \xrightarrow[]{\cancel{C_{k'},C_{k}}}1} \frac{W_{\sigma'}}{(2\beta)_{\sigma'}} + \sum_{\ell=k+1}^{m-1} W_{C_{k},C_{\ell}} \sum_{\sigma':C_{\ell} \xrightarrow[]{\cancel{C_{k'},C_{k}} }1} \frac{W_{\sigma'}}{(2\beta)_{\sigma'}}\right).
\end{aligned}\]
Since \(\sum_{\sigma':C_{k} \xrightarrow[]{\cancel{C_{k'}}} C_{k} } \frac{W_{\sigma'}}{(2\beta)_{\sigma'}} \leq G(C_{k},C_{k})\) on \(\widetilde{\Lambda}_{n,\{1,j\}}^{(k)}\), we have
\[\begin{aligned}
 &\mathbb{E}_{\widetilde{\Lambda}_{n,\{1,j\}}^{(k)}}\left(\left( \sum_{\sigma':C_{k}  \xrightarrow[]{\cancel{C_{k'}} }1} \frac{W_{\sigma'}}{(2\beta)_{\sigma'}}\right)^{s}\right)\\
 & \leq  c_s^{s}   \sum_{i\in \{0 ,\cdots,m-1,m+1 ,\cdots,n\}}W_{C_{k},B_{i}}^{s} \mathbb{E}_{\widetilde{\Lambda}_{n,\{1,j\}}^{(k)}}\left(\left(\sum_{\sigma': B_i \xrightarrow[]{\cancel{C_{k'},C_{k}}}1} \frac{W_{\sigma'}}{(2\beta)_{\sigma'}}\right)^{s}\right) \\
&+ c_s^{s}  \sum_{\ell=k+1}^{m-1} W_{C_{k},C_{\ell}}^{s} \mathbb{E}_{\widetilde{\Lambda}_{n,\{1,j\}}^{(k)}}\left(\left(\sum_{\sigma':C_{\ell} \xrightarrow[]{\cancel{C_{k'},C_{k}} }1} \frac{W_{\sigma'}}{(2\beta)_{\sigma'}}\right)^{s}\right)\\
\end{aligned}\]
Overall, we thus have
\[\begin{aligned}
& \mathbb{E}_{\widetilde{\Lambda}_{n,\{1,j\}}^{(k)}} \left(\left(\sum_{\sigma:C_{k'}\to 1} \left(\frac{W_{\sigma}}{(2\beta)_{\sigma}}\right)\right)^{s}\right) \\
& \leq  c_s^{s} (1+ c_s^{s}  W_{C_{k'},C_{k}}^{s})  \sum_{i\in \{0 ,\cdots,m-1,m+1 ,\cdots,n\}}W_{C_{k'},B_{i}}^{s} \mathbb{E}_{\widetilde{\Lambda}_{n,\{1,j\}}^{(k)}}\left(\left(\sum_{\sigma': B_i \xrightarrow[]{\cancel{C_{k'}}}1} \frac{W_{\sigma'}}{(2\beta)_{\sigma'}}\right)^{s}\right) \\
&+ c_s^{s} (1+ c_s^{s}  W_{C_{k'},C_{k}}^{s}) \sum_{\ell=k+1}^{m-1} W_{C_{k'},C_{\ell}}^{s} \mathbb{E}_{\widetilde{\Lambda}_{n,\{1,j\}}^{(k)}}\left(\left(\sum_{\sigma':C_{\ell} \xrightarrow[]{\cancel{C_{k'}} }1} \frac{W_{\sigma'}}{(2\beta)_{\sigma'}}\right)^{s}\right)
\end{aligned}.\]
We now bound the first family of path sums on the right-hand side by the Green's function \(G(B_i,1)\). Since \(G(B_i,1)=G(1,1)e^{u_{B_i}^{(1)}}\), and \(G(1,1)\) is independent of \(e^{u_{B_i}^{(1)}}\), we can factor out the term \(G(1,1)^s\) inside the expectation. The remaining expectation of \(e^{s u_{B_i}^{(1)}}\) can then be transferred to the corresponding effective model using Lemma~\ref{lem-effective-h22}. More precisely,
\[\begin{aligned}
\mathbb{E}_{\widetilde{\Lambda}_{n,\{1,j\}}^{(k)}}\left(\left(\sum_{\sigma': B_i \xrightarrow[]{\cancel{C_{k'}}}1} \frac{W_{\sigma'}}{(2\beta)_{\sigma'}}\right)^{s}\right) & \leq \mathbb{E}_{\widetilde{\Lambda}_{n,\{1,j\}}^{(k)}} \left(G(B_i,1)^{s}\right)\\
& = \mathbb{E}_{\widetilde{\Lambda}_{n,\{1,j\}}^{(k)}}(G(1,1)^{s}) \mathbb{E}_{\widetilde{\Lambda}_{n,\{1,j\}}^{(k)}}\left(\left(e^{u^{[1]}_{B_i}}\right)^{s}\right)\\
& = \mathbb{E}_{\widetilde{\Lambda}_{n,\{1,j\}}^{(k)}}(G(1,1)^{s}) \mathbb{E}_{\widetilde{\Lambda}_{n}^{(0)}}(e^{{s u^{[1]}_{B_i}}}).
\end{aligned}\]
Similarly we can treat the second path sum, therefore,
\[\begin{aligned}
& \mathbb{E}_{\widetilde{\Lambda}_{n,\{1,j\}}^{(k)}} \left(\left(\sum_{\sigma:C_{k'}\to 1} \left(\frac{W_{\sigma}}{(2\beta)_{\sigma}}\right)\right)^{s}\right) \\
& \leq  c_s^{s} (1+ c_s^{s}  W_{C_{k'},C_{k}}^{s})  \sum_{i\in \{0 ,\cdots,m-1,m+1 ,\cdots,n\}}W_{C_{k'},B_{i}}^{s} \mathbb{E}_{\widetilde{\Lambda}_{n,\{1,j\}}^{(k)}}(G(1,1)^{s}) \mathbb{E}_{\widetilde{\Lambda}_{n}^{(0)}}(e^{{s u^{[1]}_{B_i}}}) \\
&+ c_s^{s} (1+ c_s^{s}  W_{C_{k'},C_{k}}^{s}) \sum_{\ell=k+1}^{m-1} W_{C_{k'},C_{\ell}}^{s} \mathbb{E}_{\widetilde{\Lambda}_{n,\{1,j\}}^{(k)}}\left(G(1,1)^{s}\right)\mathbb{E}_{\widetilde{\Lambda}_{n,\{1,j\}}^{(\ell)}}(e^{{s u^{[1]}_{C_{\ell}}}}).
\end{aligned}\]
Overall we have the following recursive inequality:
 \begin{equation}\begin{aligned}
\label{eq-recursive-inequality-Gij}
& \mathbb{E}_{\widetilde{\Lambda}_{n,\{1,j\}}^{(k)}}(e^{{s u^{[1]}_{C_{k'}}}}) \\
& \leq  c_s^{s} (1+ c_s^{s}  W_{C_{k'},C_{k}}^{s})  \sum_{i\in \{0 ,\cdots,m-1,m+1 ,\cdots,n\}}W_{C_{k'},B_{i}}^{s}  \mathbb{E}_{\widetilde{\Lambda}_{n}^{(0)}}(e^{{s u^{[1]}_{B_i}}}) \\
&+ c_s^{s} (1+ c_s^{s}  W_{C_{k'},C_{k}}^{s}) \sum_{\ell=k+1}^{m-1} W_{C_{k'},C_{\ell}}^{s} \mathbb{E}_{\widetilde{\Lambda}_{n,\{1,j\}}^{(\ell)}}(e^{{s u^{[1]}_{C_{\ell}}}}).
\end{aligned}\end{equation}
To compare \eqref{eq-recursive-inequality-Gij} with the one-point recursion \eqref{eq-recursion-fmm}, note from \eqref{eq-block-W-CkClBi} that
\[
W_{C_{k'},C_k}=W_{B_{k'},B_k}, \qquad W_{C_{k'},C_\ell}=W_{B_{k'},B_\ell}=W_{B_k,B_\ell} \quad (\ell\ge k+1).
\]
The contribution of the \(B_i\)-terms plays the role of the terminal contribution in the one-point recursion, and we collect it into a single remainder by setting
\[
\mathfrak C\,W_{B_k,B_m}^s:=\sum_{i\in \{0,\dots,m-1,m+1,\dots,n\}}W_{C_{k'},B_i}^s\,\mathbb E_{\widetilde\Lambda_n^{(0)}}\bigl(e^{s u_{B_i}^{[1]}}\bigr).
\]
With this notation, \eqref{eq-recursive-inequality-Gij} becomes
\[\begin{aligned}
& \mathbb{E}_{\widetilde{\Lambda}_{n,\{1,j\}}^{(k)}}(e^{{s u^{[1]}_{C_{k'}}}}) \\
& \leq c_s^{s}(1+c_s^{s}W_{B_{k'},B_k}^{s})\left[\sum_{\ell=k+1}^{m-1}W_{B_{k'},B_\ell}^{s}\,\mathbb E_{\widetilde\Lambda_{n,\{1,j\}}^{(\ell)}}\bigl(e^{s u_{C_\ell}^{[1]}}\bigr)+W_{B_k,B_m}^{s}\,\mathfrak C\right].
\end{aligned}\]
By Proposition~\ref{prop-G-1-B-m}, the quantity \(\mathfrak C\) is bounded uniformly in \(n\). We may therefore iterate this inequality exactly as in the proof of Proposition~\ref{prop-2-FMMbound-uBk}: each step replaces \(k\) by a larger index until the chain reaches \(m\). This gives
\begin{equation}\begin{aligned}
\label{eq-tobeinserted}
\mathbb{E}_{\widetilde{\Lambda}_{n,\{1,j\}}^{(k)}}(e^{{s u^{[1]}_{C_{k'}}}}) \leq \mathfrak{C} \sum_{k=k_0 < k_1<k_2 <\cdots <k_{\ell}=m } \prod_{i=0}^{\ell-1}  \left[\left(1+ c_s^{s} W_{B_{k_i'},B_{k_i}}^{s}\right)c_s^{s} W_{B_{k_i},B_{k_{i+1}}}^{s}\right]
\end{aligned}\end{equation}
Assume for the moment that \(d<2\). Splitting the definition of \(\mathfrak C\) into the three ranges \(i<m\), \(m<i<n\), and \(i=n\), and using Proposition~\ref{prop-G-1-B-m} together with \eqref{eq-block-W-CkClBi}, we obtain
\[\begin{aligned}
\mathfrak C
&\le C''(\overline W,d,s)\Bigg[
2^{-ms}\sum_{i=0}^{m-1}\left(\frac{2}{\rho}\right)^{is}
+\rho^{-ms}\sum_{r=1}^{n-m-1}\rho^{-2sr}
+(\rho-1)^{-s}\rho^{-s(2n-m-1)}
\Bigg].
\end{aligned}\]
In particular, \(\mathfrak C\) is bounded uniformly in \(n\).

Inserting this into \eqref{eq-tobeinserted} and then using \eqref{eq-sum-one-path-1plusAj}, we obtain
\[\begin{aligned}
&\mathbb{E}_{\widetilde{\Lambda}_{n,\{1,j\}}^{(k)}}(e^{{s u^{[1]}_{C_{k'}}}})\leq \mathfrak{C} C(\overline{W},d,s) \rho^{-{s (m-k)}}  \leq  C(\overline{W},d,s)(2\rho)^{-{m s}}\rho^{s k}.
\end{aligned}\]
Setting \(k=0\) gives \eqref{eq-FMM-eu-subcritical-ij}.

When \(d=2\), the same split gives
\[\begin{aligned}
\mathfrak C
&\le C''(\overline W,s)\Bigg[
2^{-ms}\sum_{i=0}^{m-1}\left(\frac{2}{c(\overline W,s)}\right)^{is}
+c(\overline W,s)^{-ms}\sum_{r=1}^{n-m-1}\left(\frac{1}{2c(\overline W,s)}\right)^{sr}
+(2c(\overline W,s))^{-ns}
\Bigg].
\end{aligned}\]
So again \(\mathfrak C\) is bounded uniformly in \(n\). We then choose \(\overline{W}\) small enough so that \eqref{eq-choose-upsilon} holds, and the same summation argument yields the desired geometric decay across hierarchical scales.
\end{proof}

\begin{proof}[Proof of Proposition \ref{prop-G-1-B-m}]
The case \(\ell=n\) was already proved in Theorem~\ref{thm-FMM-exp-u}, since \(B_{\ell}=B_n=\delta_n\) (also we use \eqref{eq-euij-euji-are-equal}), the case \(\ell=0,1\) follows the Ward identity \eqref{eq-ward-identity-eu=1} and Jensen inequality. We therefore restrict attention to \(2\le \ell\le n-1\). Fix \(0\le k<j\le n\) and consider the coarse graph \(\widetilde{\Lambda}_n^{(k)}\). Then
\begin{equation}\begin{aligned}
\label{eq-A-1-Gij}
  \mathbb{E}_{\widetilde{\Lambda}_{n}^{(k)}}(G(B_j,B_j)^{s}) \mathbb{E}_{\widetilde{\Lambda}_{n}^{(k)}}(e^{{s u^{[B_{j}]}_{B_{k'}}}}) & = \mathbb{E}_{\widetilde{\Lambda}_{n}^{(k)}}(G(B_{k'},B_j)^{s})\\
& = \mathbb{E}_{\widetilde{\Lambda}_{n}^{(k)}}\left(\left(\sum_{\sigma:B_{k'}\to B_j} \left( \frac{W_{\sigma}}{(2 \beta)_{\sigma}}\right)\right)^{s}\right)
\end{aligned}\end{equation}
where the sum is over the set \(B_{k'} \to B_j\) of all paths \(\sigma\) going from \(B_{k'}\) to \(B_j\) in the graph \(\widetilde{\Lambda}_{n}^{(k)}\).

We start by proving a counterpart of Proposition \ref{prop-1-recursiveBdd-uBk} in our \(G(i,j)\) setting instead of the previous simpler \(G(i,\delta)\) setting.

We split the path sum according to the first jump after the last visit to \(B_{k'}\). Either the path next jumps to \(B_k\), in which case it belongs to \(B_{k'}\Rightarrow B_k\to B_j\), or it jumps directly to some block \(B_\ell\) with \(\ell>k\), in which case it belongs to \(B_{k'}\Rightarrow B_{>k}\to B_j\). Thus
\begin{equation}\begin{aligned}
\label{eq-pathsum-2-parts-Gij}
\sum_{\sigma:B_{k'}\to B_j} \left( \frac{W_{\sigma}}{(2 \beta)_{\sigma}}\right)= \sum_{\sigma:B_{k'}\Rightarrow B_{k}\to B_j} \left( \frac{W_{\sigma}}{(2 \beta)_{\sigma}}\right)+ \sum_{\sigma:B_{k'}\Rightarrow B_{>k}\to B_j} \left( \frac{W_{\sigma}}{(2 \beta)_{\sigma}}\right).
\end{aligned}\end{equation}
For the first sum (over the set \(B_{k'} \Rightarrow B_{k} \to B_j\)) in \eqref{eq-pathsum-2-parts-Gij}, as before, we cut the path into 3 parts, then we have

\begin{equation*}\begin{aligned}
&\sum_{B_{k'} \Rightarrow B_{k} \to  B_j } \frac{W_{\sigma}}{(2\beta)_{\sigma}}\\
& = G(B_{k'},B_{k'}) W_{B_{k'},B_{k}} \sum_{\sigma':B_{k} \xrightarrow[]{\cancel{B_{k'}}} B_{k}} \frac{W_{\sigma'}}{(2\beta)_{\sigma'}} \sum_{\ell=k+1}^{n} W_{B_{k},B_{\ell}} \sum_{\sigma': B_{\ell} \xrightarrow[]{\cancel{B_{k'},B_{k}}} B_j} \frac{W_{\sigma'}}{(2\beta)_{\sigma'}}
\end{aligned}\end{equation*}
where \(\sigma':B_{k} \xrightarrow[]{\cancel{B_{k'}}} B_{k}\) means the set of paths going from \(B_{k}\) to \(B_{k}\) without using the vertex \(B_{k'}\); similarly, \(B_{\ell} \xrightarrow[]{\cancel{B_{k'},B_{k}}} B_j\) denotes the set of paths going from \(B_{\ell}\) to \(B_j\) without visiting \(\{B_{k'},B_{k}\}\).

Under the law \(\mathbb{E}_{\widetilde{\Lambda}_{n}^{(k)}}\), the factor \(G(B_{k'},B_{k'})\) depends only on the diagonal variable at \(B_{k'}\), whereas the remaining path sums are measurable with respect to the field away from \(B_{k'}\) after its last visit. By Lemma~\ref{lem-gamma-dist-Gii}, \(G(B_{k'},B_{k'})\) is independent of the remaining factor. Hence
\begin{equation*}\begin{aligned}
&\mathbb{E}_{\widetilde{\Lambda}_{n}^{(k)}}\left( \left(\sum_{B_{k'} \Rightarrow B_{k} \to  B_j } \frac{W_{\sigma}}{(2\beta)_{\sigma}}\right)^{s}\right) \\
& = c_s^{s} \mathbb{E}_{\widetilde{\Lambda}_{n}^{(k)}}\left( \left(W_{B_{k'},B_{k}} \sum_{\sigma':B_{k} \xrightarrow[]{\cancel{B_{k'}}} B_{k}} \frac{W_{\sigma'}}{(2\beta)_{\sigma'}} \sum_{\ell=k+1}^{n} W_{B_{k},B_{\ell}} \sum_{\sigma': B_{\ell} \xrightarrow[]{\cancel{B_{k'},B_{k}}} B_j} \frac{W_{\sigma'}}{(2\beta)_{\sigma'}}\right)^{s}\right)
\end{aligned}\end{equation*}

Next
\begin{equation*}\begin{aligned}
& \sum_{\sigma':B_{k} \xrightarrow[]{\cancel{B_{k'}}} B_{k}} \frac{W_{\sigma'}}{(2\beta)_{\sigma'}} \sum_{\ell=k+1}^{n} W_{B_{k},B_{\ell}} \sum_{\sigma': B_{\ell} \xrightarrow[]{\cancel{B_{k'},B_{k}}} B_j} \frac{W_{\sigma'}}{(2\beta)_{\sigma'}} \\
&\leq   \sum_{\sigma':B_{k} \xrightarrow[]{} B_{k}} \frac{W_{\sigma'}}{(2\beta)_{\sigma'}} \sum_{\ell=k+1}^{n} W_{B_{k},B_{\ell}} \sum_{\sigma': B_{\ell} \xrightarrow[]{\cancel{B_{k'},B_{k}}} B_j} \frac{W_{\sigma'}}{(2\beta)_{\sigma'}}\\
& =G(B_{k},B_{k}) \sum_{\ell=k+1}^{n} W_{B_{k},B_{\ell}} \sum_{\sigma': B_{\ell} \xrightarrow[]{\cancel{B_{k'},B_{k}}} B_j} \frac{W_{\sigma'}}{(2\beta)_{\sigma'}}
\end{aligned}\end{equation*}
and again \(G(B_{k},B_{k})\) depends only on the diagonal variable at \(B_k\), while the remaining factor is measurable away from \(B_k\). By Lemma~\ref{lem-gamma-dist-Gii}, it is independent of that factor, so
\begin{equation*}\begin{aligned}
&\mathbb{E}_{\widetilde{\Lambda}_{n}^{(k)}}\left( \left(\sum_{\sigma':B_{k} \xrightarrow[]{\cancel{B_{k'}}} B_{k}} \frac{W_{\sigma'}}{(2\beta)_{\sigma'}} \sum_{\ell=k+1}^{n} W_{B_{k},B_{\ell}} \sum_{\sigma': B_{\ell} \xrightarrow[]{\cancel{B_{k'},B_{k}}} B_j} \frac{W_{\sigma'}}{(2\beta)_{\sigma'}} \right)^{s} \right) \\
& \leq c_s^{s} \sum_{\ell=k+1}^{n} W_{B_{k},B_{\ell}}^{s} \mathbb{E}_{\widetilde{\Lambda}_{n}^{(k)}}\left(  \left(  \sum_{\sigma': B_{\ell} \xrightarrow[]{\cancel{B_{k'},B_{k}}} B_j} \frac{W_{\sigma'}}{(2\beta)_{\sigma'}}\right)^{s} \right)\\
& \leq c_s^{s} \sum_{\ell=k+1}^{n} W_{B_{k},B_{\ell}}^{s} \mathbb{E}_{\widetilde{\Lambda}_{n}^{(k)}}\left(  \left(  \sum_{\sigma': B_{\ell} \xrightarrow[]{} B_j} \frac{W_{\sigma'}}{(2\beta)_{\sigma'}}\right)^{s} \right)
\end{aligned}\end{equation*}
After dropping the avoidance constraint, the remaining factor is exactly
\begin{equation}\begin{aligned}
\label{eq-Bl-deltan-Geu-Gij}
\mathbb{E}_{\widetilde{\Lambda}_{n}^{(k)}}\left( \left(\sum_{\sigma':B_{\ell} \xrightarrow[]{} B_j} \frac{W_{\sigma'}}{(2 \beta)_{\sigma'}}\right)^{s}
\right) & = \mathbb{E}_{\widetilde{\Lambda}_{n}^{(k)}}\left(G(B_j,B_{\ell})^{s}\right)\\
& = \mathbb{E}_{\widetilde{\Lambda}_{n}^{(k)}}(G(B_j,B_j)^{s}) \, \mathbb{E}_{\widetilde{\Lambda}_{n}^{(k)}}\left(e^{s u^{[B_j]}_{B_{\ell}}}\right)\\
& = \begin{cases} \mathbb{E}_{\widetilde{\Lambda}_{n}^{(k)}}(G(B_j,B_j)^{s}) \, \mathbb{E}_{\widetilde{\Lambda}_{n}^{(\ell)}}\left(e^{s u^{[B_j]}_{B_{\ell}}}\right), & \ell\le j, \\ \mathbb{E}_{\widetilde{\Lambda}_{n}^{(k)}}(G(B_j,B_j)^{s}) \, \mathbb{E}_{\widetilde{\Lambda}_{n}^{(j)}}\left(e^{s u^{[B_j]}_{B_{\ell}}}\right), & \ell >j,
\end{cases}
\end{aligned}\end{equation}
by Corollary~\ref{coro-effective-model-u}.

To summarize, our discussion of the first sum gives the following upper bound:
\begin{equation}\begin{aligned}
\label{eq-3A-Gij}
&\mathbb{E}_{\widetilde{\Lambda}_{n}^{(k)}}\left( \left(\sum_{B_{k'} \Rightarrow B_{k} \to  B_j } \frac{W_{\sigma}}{(2\beta)_{\sigma}}\right)^{s}\right) \\
&\leq c_s^{{2 s}} W_{B_{k'},B_{k}}^{s} \sum_{\ell=k+1}^{j} W_{B_k,B_{\ell}}^{s} \mathbb{E}_{\widetilde{\Lambda}_{n}^{(k)}}(G(B_j,B_j)^{s}) \mathbb{E}_{\widetilde{\Lambda}_{n}^{(\ell)}}\left(e^{{s u^{[B_j]}_{B_{\ell}}}}\right) \\
& +c_s^{{2 s}} W_{B_{k'},B_{k}}^{s} \sum_{\ell=j+1}^{n} W_{B_k,B_{\ell}}^{s} \mathbb{E}_{\widetilde{\Lambda}_{n}^{(k)}}(G(B_j,B_j)^{s}) \mathbb{E}_{\widetilde{\Lambda}_{n}^{(j)}}\left(e^{{s u^{[B_j]}_{B_{\ell}}}}\right) 
\end{aligned}\end{equation}

For the second class of paths, we factorize the initial segment up to the last visit to \(B_{k'}\), which contributes the factor \(G(B_{k'},B_{k'})\). Thus
\begin{equation*}\begin{aligned}
\sum_{B_{k'} \Rightarrow B_{>k} \to B_j} \frac{W_{\sigma}}{(2\beta)_{\sigma}} = G(B_{k'},B_{k'}) \sum_{\ell=k+1}^{n} W_{B_{k'},B_{\ell}} \sum_{\sigma':B_{\ell} \xrightarrow[]{\cancel{B_{k'}}} B_j}  \frac{W_{\sigma'}}{(2 \beta)_{\sigma'}}
\end{aligned}\end{equation*}
Under the law \(\mathbb{E}_{\widetilde{\Lambda}_{n}^{(k)}}\), \(G(B_{k'},B_{k'})\) is independent of \(\sum_{\sigma':B_{\ell} \xrightarrow[]{\cancel{B_{k'}}} B_j}  \frac{W_{\sigma'}}{(2 \beta)_{\sigma'}}\)  and it equals in law to \(\frac{1}{2\gamma}\), therefore,
\begin{equation*}\begin{aligned}
&\mathbb{E}_{\widetilde{\Lambda}_{n}^{(k)}}\left(\left(\sum_{B_{k'} \Rightarrow B_{>k} \to B_j} \frac{W_{\sigma}}{(2\beta)_{\sigma}}\right)^{s}\right) \\
& \le \sum_{\ell=k+1}^{n} W_{B_{k'},B_{\ell}}^{s}  \mathbb{E}_{\widetilde{\Lambda}_{n}^{(k)}}\left(G(B_{k'},B_{k'})^{s}\right) \mathbb{E}_{\widetilde{\Lambda}_{n}^{(k)}}\left( \left(\sum_{\sigma':B_{\ell} \xrightarrow[]{\cancel{B_{k'}}} B_j}  \frac{W_{\sigma'}}{(2 \beta)_{\sigma'}}\right)^{s}
\right)
\\ &\leq  \sum_{\ell=k+1}^{n} W_{B_{k'},B_{\ell}}^{s}  \mathbb{E}_{\widetilde{\Lambda}_{n}^{(k)}}\left(G(B_{k'},B_{k'})^{s}\right) \mathbb{E}_{\widetilde{\Lambda}_{n}^{(k)}}\left( \left(\sum_{\sigma':B_{\ell} \xrightarrow[]{} B_j}  \frac{W_{\sigma'}}{(2 \beta)_{\sigma'}}\right)^{s}
\right)\\
& =  \sum_{\ell=k+1}^{n} W_{B_{k'},B_{\ell}}^{s}  c_s^{s}\mathbb{E}_{\widetilde{\Lambda}_{n}^{(k)}}\left( \left(\sum_{\sigma':B_{\ell} \xrightarrow[]{} B_j}  \frac{W_{\sigma'}}{(2 \beta)_{\sigma'}}\right)^{s}
\right)
\end{aligned}\end{equation*}
The last factor was identified in \eqref{eq-Bl-deltan-Geu-Gij}. Therefore the second sum satisfies the following upper bound:
\begin{equation}\begin{aligned}
\label{eq-4A-Gij}
\mathbb{E}_{\widetilde{\Lambda}_{n}^{(k)}}\left(\left(\sum_{B_{k'} \Rightarrow B_{>k} \to B_j} \frac{W_{\sigma}}{(2\beta)_{\sigma}}\right)^{s}\right) & \leq  \sum_{\ell=k+1}^{j} W_{B_{k'},B_{\ell}}^{s}  c_s^{s} \mathbb{E}_{\widetilde{\Lambda}_{n}^{(k)}}(G(B_j,B_j)^{s}) \mathbb{E}_{\widetilde{\Lambda}_{n}^{(\ell)}}\left(e^{{s u^{[B_j]}_{B_{\ell}}}}\right) \\
& + \sum_{\ell=j+1}^{n} W_{B_{k'},B_{\ell}}^{s}  c_s^{s} \mathbb{E}_{\widetilde{\Lambda}_{n}^{(k)}}(G(B_j,B_j)^{s}) \mathbb{E}_{\widetilde{\Lambda}_{n}^{(j)}}\left(e^{{s u^{[B_j]}_{B_{\ell}}}}\right) 
\end{aligned}\end{equation}

Combining Eqs \eqref{eq-A-1-Gij}, \eqref{eq-pathsum-2-parts-Gij}, \eqref{eq-3A-Gij}, \eqref{eq-4A-Gij}, and dividing both sides by \(\mathbb{E}_{\widetilde{\Lambda}_{n}^{(k)}}(G(B_j,B_j)^{s})=c_s^s\), we obtain
\begin{equation}\begin{aligned}
\label{eq-recursion-fmm-Gij}
  & \mathbb{E}_{\widetilde{\Lambda}_{n}^{(k)}}(e^{{s u^{[B_j]}_{B_{k'}}}})  \\
& \leq  c_s^{{2 s}} W_{B_{k'},B_{k}}^{s} \sum_{\ell=k+1}^{n} W_{B_k,B_{\ell}}^{s} \mathbb{E}_{\widetilde{\Lambda}_{n}^{(\ell)}}\left(e^{{s u^{[B_j]}_{B_{\ell}}}}\right)  \\
& + \sum_{\ell=k+1}^{n} W_{B_{k'},B_{\ell}}^{s}  c_s^{s} \mathbb{E}_{\widetilde{\Lambda}_{n}^{(\ell)}}\left(e^{{s u^{[B_j]}_{B_{\ell}}}}\right) \\
& = \left(1+ c_s^{s} W_{B_{k'},B_{k}}^{s}\right) c_s^{s} \left(\sum_{\ell=k+1}^{j} W_{B_{k},B_{\ell}}^{s} \mathbb{E}_{\widetilde{\Lambda}_{n}^{(\ell)}}(e^{{s u^{[B_j]}_{B_{\ell}}}})+ \sum_{\ell=j+1}^{n} W_{B_{k},B_{\ell}}^{s} \mathbb{E}_{\widetilde{\Lambda}_{n}^{(j)}}(e^{{s u^{[B_j]}_{B_{\ell}}}})\right).
\end{aligned}\end{equation}
This recursive inequality is the counterpart of Proposition \ref{prop-1-recursiveBdd-uBk} in the case of \(G(i,j)\).

Equation~\eqref{eq-recursion-fmm-Gij} is the two-point analogue of the one-point recursion. It shows that the fractional moment of \(G(B_{k'},B_j)\) can be propagated by moving the left endpoint to the right, and once it overshoots \(j\), by continuing from the right. At each step the relevant graph is more coarse-grained, so the entropy of the remaining path expansion decreases; see Figure~\ref{fig:tikzfiggij2}.

We now iterate \eqref{eq-recursion-fmm-Gij} on the terms \(G(B_{i_1},B_j)=G(B_j,B_j)e^{u^{[B_j]}_{B_{i_1}}}\) and \(G(B_j,B_{i_2})=G(B_j,B_j)e^{u^{[B_j]}_{B_{i_2}}}\). A convenient way to picture the iteration is with two particles: \(A\), starting at \(k'\), and \(B\), starting at \(j\). Particle \(A\) moves first:
\begin{itemize}
\item If \(A\) jumps to a position \(k_1<j\), we update \(A\) to \(k_1\) and let \(A\) move again.
\item If \(A\) jumps to a position \(k_2>j\), we update \(A\) to \(k_2\) and then let \(B\) move.
\item If \(B\) jumps past \(A\), we update \(B\) and let \(A\) move again; otherwise \(B\) continues moving to the right.
\end{itemize}
The iteration stops when both particles meet at some index \(m\). Summing over all admissible jump sequences gives
\[\sum_{\substack{k=k_0<k_1<\cdots <k_{\ell_{1}}=m \\ j=j_0<j_{1}<\cdots < j_{\ell_{2}}=m\\ \{k_{i}\}_{i=0}^{\ell_{1}}\cap \{j_{i}\}_{i=0}^{\ell_{2}}=\{m\}}}\]
where \(\{k_0, \dots, k_{\ell_1}\}\) are the successive positions of \(A\), and \(\{j_0, \dots, j_{\ell_2}\}\) are the successive positions of \(B\). The only common position is \(m\), where they finally meet.
\begin{figure}
\begin{center}
\begin{tikzpicture}[scale=1.2]
    \foreach \i in {0,1,3,5,7,9} {
        \node[fill=black, circle, inner sep=1pt] (p\i) at (\i,0) {};
    }
    
    \node[below] at (p0) {\(k'\)};
    \node[below] at (p1) {\(k\)};
    \node[below] at (p3) {\(i_1\)};
    \node[below] at (p5) {\(j\)};
    \node[below] at (p7) {\(i_2\)};
    \node[below] at (p9) {\(n\)};

    \node at (2,0) {\(\dots\)};
    \node at (4,0) {\(\dots\)};
    \node at (6,0) {\(\dots\)};
    \node at (8,0) {\(\dots\)};

    \draw[->] (p0) to[out=20, in=160] (p3); 
    \draw[->] (p0) to[out=20, in=160] (p7); 

    \draw (p0) -- (0,1.5); 
    \draw (5,0.05) -- (5,1.5); 
    \draw (0,1.5) -- (5,1.5); 
    \node[above] at (2.5,1.5) {\(G(B_{k'}, B_j)\)}; 

    \draw (5,-0.5) -- (4.9,-1.5); 
    \draw (3,-0.5) -- (3,-1.5); 
    \draw (3,-1.5) -- (4.9,-1.5); 
    \node[below] at (4,-1.5) {\(G(B_{i_1 }, B_j)\)}; 

    \draw (7,-0.5) -- (7,-2); 
    \draw (5,-0.5) -- (5.1,-2); 
    \draw (5.1,-2) -- (7,-2); 
    \node[below] at (6,-2) {\(G(B_j, B_{i_2})\)}; 

\end{tikzpicture}
\caption{\emph{Illustration of the inequality \eqref{eq-recursion-fmm-Gij}, where fractional moment of \(G(B_{k'},B_j) \) is bounded by sum over \(G(B_{i_1},B_j) \) and \(G(B_j,B_{i_2}) \)}.}
\end{center}
\label{fig:tikzfiggij2}
\end{figure}

By successive iterations we therefore obtain a bound summing over all admissible jump sequences, where \(k_0,k_1,\dots,k_{\ell_1}\) are the successive positions of particle \(A\), \(j_0,j_1,\dots,j_{\ell_2}\) are those of particle \(B\), and both sequences stop once the particles meet at the same index \(m\). The resulting estimate is
\begin{equation*}\begin{aligned}
&\mathbb{E}_{\widetilde{\Lambda}_{n}^{(k)}}(e^{{s u^{[B_{j}]}_{B_{k}}}})  \\
& \leq  \left(1+ c_s^{s} W_{B_{k'},B_{k}}^{s}\right) c_s^{s} \left(\sum_{k_1=k+1}^{j} W_{B_{k},B_{k_1}}^{s} \mathbb{E}_{\widetilde{\Lambda}_{n}^{(k_1)}}(e^{{s u_{B_{k_1}}}})+ \sum_{k_1=j+1}^{n} W_{B_{k},B_{k_1}}^{s} \mathbb{E}_{\widetilde{\Lambda}_{n}^{(j)}}(e^{{s u_{B_{k_1}}}})\right) \\
& \vdots \\
& \leq  \sum_{m=j}^{n} \sum_{\substack{k=k_0<k_1<\cdots <k_{\ell_{1}}=m \\ j=j_0<j_{1}<\cdots < j_{\ell_{2}}=m\\ \{k_{i}\}_{i=0}^{\ell_{1}}\cap \{j_{i}\}_{i=0}^{\ell_{2}}=\{m\}}} \prod_{i=0}^{\ell_{1}-1} \left[\left(1+ c_s^{s} W_{B_{k_i'},B_{k_i}}^{s}\right)c_s^{s} W_{B_{k_i},B_{k_{i+1}}}^{s}\right]\prod_{i=0}^{\ell_{2}-1} \left[\left(1+ c_s^{s} W_{B_{j_i'},B_{j_i}}^{s}\right)c_s^{s} W_{B_{j_i},B_{j_{i+1}}}^{s}\right]
\end{aligned}\end{equation*}
Then we upper bound this sum by summing over the \(\{k_0 ,\cdots,k_{\ell_{1}}\}\) and \(\{j_0 ,\cdots,j_{\ell_{2}}\}\) without the condition that the only intersection is \(m\),
\[\begin{aligned}
& \leq  \sum_{m=j}^{n} \sum_{\substack{k=k_0<k_1<\cdots <k_{\ell_{1}}=m \\ j=j_0<j_{1}<\cdots < j_{\ell_{2}}=m}} \prod_{i=0}^{\ell_{1}-1} \left[\left(1+ c_s^{s} W_{B_{k_i'},B_{k_i}}^{s}\right)c_s^{s} W_{B_{k_i},B_{k_{i+1}}}^{s}\right]\prod_{i=0}^{\ell_{2}-1} \left[\left(1+ c_s^{s} W_{B_{j_i'},B_{j_i}}^{s}\right)c_s^{s} W_{B_{j_i},B_{j_{i+1}}}^{s}\right]\\
& \leq \sum_{m=j}^{n} \left(\sum_{{k=k_0<k_1<\cdots <k_{\ell_{1}}=m }} \prod_{i=0}^{\ell_{1}-1} \left[\left(1+ c_s^{s} W_{B_{k_i'},B_{k_i}}^{s}\right)c_s^{s} W_{B_{k_i},B_{k_{i+1}}}^{s}\right]\right) \\
&\times \left(\sum_{ j=j_0<j_{1}<\cdots < j_{\ell_{2}}=m}\prod_{i=0}^{\ell_{2}-1} \left[\left(1+ c_s^{s} W_{B_{j_i'},B_{j_i}}^{s}\right)c_s^{s} W_{B_{j_i},B_{j_{i+1}}}^{s}\right]\right).
\end{aligned}\]
Using \eqref{eq-sum-one-path-1plusAj}, we obtain, in the case \(d<2\),
\[\begin{aligned}
\mathbb{E}_{\widetilde{\Lambda}_{n}^{(k)}}(e^{{s u^{[B_{j}]}_{B_{k}}}})
&\leq \sum_{m=j}^{n} C(\overline{W},d,s)\rho^{-s(m-k)}\, C(\overline{W},d,s)\rho^{-s(m-j)}\\
&= C(\overline{W},d,s)^2\rho^{s(k-j)}\frac{1-\rho^{-2s(n-j+1)}}{1-\rho^{-2s}}.
\end{aligned}\]
Setting \(k=0\), we may take
\[
C''(\overline{W},d,s)=C(\overline{W},d,s)^2\frac{1}{1-\rho^{-2s}}.
\]

When \(d=2\), we instead use \eqref{eq-sum-one-path-1plusAj-d-is-2}. Choosing \(\overline{W}\) small enough so that \eqref{eq-choose-upsilon} holds, we conclude exactly as in the proof of Theorem~\ref{thm-FMM-exp-u}, with \(c(\overline{W},s)\) in place of \(\rho\).
\end{proof}

%
%


\section*{Acknowledgments}
We thank M. Disertori for helpful discussions on the first draft of this manuscript. We also acknowledge the support of the Institut de recherche en mathématiques, interactions \& applications: IRMIA\(++\).

\end{document}